\documentclass[a4paper,dvipsnames]{article}

\usepackage{amsmath,amssymb,amsthm,mathtools}
\usepackage{cite,enumitem,upgreek}

\usepackage{tabularx}
\newcolumntype{P}[1]{>{\centering\arraybackslash}p{#1}}
\newcolumntype{N}{@{}m{0pt}@{}}

\usepackage[hypcap=false]{caption}
\usepackage{subcaption}
\usepackage[pdftex,bookmarksopen=true,bookmarks=true,unicode,setpagesize]{hyperref}
\hypersetup{colorlinks=true,linkcolor=black,citecolor=black,urlcolor=orange}

\usepackage{pgf,tikz}
\usetikzlibrary{arrows,patterns,decorations.text,shapes.geometric,decorations.pathreplacing}

\allowdisplaybreaks[3]

\numberwithin{equation}{section}

\let\equation=\gather
\let\endequation=\endgather

\makeatletter
\renewcommand*{\@fnsymbol}[1]{\ensuremath{\ifcase#1\or 1\or 2\or
   3\else\@ctrerr\fi}}
\makeatother

\pgfdeclaredecoration{arrowborder}{tick}
{
  \state{tick}[switch if less than=+\pgfdecorationsegmentlength to last,
               width=+\pgfdecorationsegmentlength]
  {
    \pgfsetarrows{latex-} 
    \pgfpathmoveto{\pgfpointorigin}
    \pgfpathlineto{\pgfpointpolar{\pgfdecorationsegmentangle}{+\pgfdecorationsegmentamplitude}}
     \pgfusepath{stroke} 
  }
  \state{last}[width=+\pgfdecorationsegmentamplitude,next state=final]
  {
    \pgfsetarrows{latex-} 
    \pgfpathmoveto{\pgfpointorigin}
    \pgfpathlineto{\pgfpointpolar{\pgfdecorationsegmentangle}{+\pgfdecorationsegmentamplitude}}
     \pgfusepath{stroke} 
  }  
  \state{final}
  {
    \pgfpathmoveto{\pgfpointdecoratedpathlast}
  }
}

\newtheorem{theorem}{Theorem}[section]

\newtheorem{lemma}[theorem]{Lemma}
\newtheorem{proposition}[theorem]{Proposition}
\theoremstyle{definition}
\newtheorem{definition}[theorem]{Definition}
\newtheorem{example}[theorem]{Example}
\theoremstyle{remark}
\newtheorem{remark}[theorem]{Remark}

\DeclareMathOperator*{\esssup}{ess\,sup}
\DeclareMathOperator*{\essinf}{ess\,inf}
\newcommand{\R}{\mathbb{R}}
\newcommand{\X}{{\mathbb{R}^d}}
\newcommand{\N}{\mathbb{N}}
\newcommand{\eps}{\varepsilon}
\newcommand{\la}{\lambda}
\newcommand{\La}{\Lambda}
\newcommand{\ga}{\gamma}
\newcommand{\ka}{\varkappa}

\newcommand\F[1][]{\mathcal{F}_{#1}}
\newcommand{\tomega}{{\widetilde{\upomega}}}
\newcommand\Dt{{\mathcal{D}_{d}}}
\newcommand\m[1]{{\langle #1\rangle}}
\newcommand\Et[1][d]{\mathcal{E}_{#1}}
\newcommand{\1}{1\!\!1}
\newcommand{\normom}[1]{\lVert #1\rVert_\tomega}
\newcommand\dt{{\dfrac{\partial}{\partial t}}}
\newcommand{\locun}{\xRightarrow{\,\mathrm{loc}\ }}
\let\emptyset=\varnothing

\newcounter{assum}
\newenvironment{assum}{\refstepcounter{assum}\let\myabovedisplayskip=\abovedisplayskip\abovedisplayskip=8pt plus 2pt minus 2pt\let\mybelowdisplayskip=\belowdisplayskip\belowdisplayskip=8pt plus 2pt minus 2pt\equation\tag{\ensuremath{\mathrm{A}\theassum}}}{\endequation\let\abovedisplayskip=\myabovedisplayskip\let\belowdisplayskip=\mybelowdisplayskip}

\title{Accelerated front propagation for monostable equations with~nonlocal~diffusion: Multidimensional case}
\author{Dmitri Finkelshtein\thanks{Department of Mathematics,
Swansea University, Singleton Park, Swansea SA2 8PP, U.K. ({\tt d.l.finkelshtein@swansea.ac.uk}).} \and Yuri Kondratiev\thanks{Fakult\"{a}t
f\"{u}r Mathematik, Universit\"{a}t Bielefeld, Postfach 110 131, 33501 Bielefeld,
Germany ({\tt kondrat@math.uni-bielefeld.de}).} \and Pasha Tkachov\thanks{Gran Sasso Science Institute, Viale Francesco Crispi, 7, 67100 L'Aquila AQ, Italy ({\tt pasha.tkachov@gssi.it}).}}

\begin{document}

\maketitle

\begin{abstract}
We describe acceleration of the front propagation for solutions to a class of monostable nonlinear equations with a nonlocal diffusion in $\X$, $d\geq1$. We show that the acceleration takes place if either the diffusion kernel or the initial condition has `regular' heavy tails in $\X$ (in particular, decays slower than exponentially). Under general assumptions which can be verified for particular models, we present sharp estimates for the time-space zone which separates the region of convergence to the unstable zero solution with the region of convergence to the stable positive constant solution.  We show the variety of different possible rates of the propagation starting from a little bit faster than a linear one up to the exponential rate.  The paper generalizes to the case $d>1$ our results for the case $d=1$ obtained early~in~\cite{FT2017c}.


\end{abstract}

\tableofcontents

\section{Introduction}

\subsection{Object of study}
The present paper is aimed to study the accelerated propagation of the front for a non-negative solutions $u:\X\times\R_+\to\R_+:=[0,\infty)$, $d\geq1$, to the equation
\begin{equation}\label{eq:basicequation}
\begin{cases}
\dt u(x,t) = \ka \displaystyle \int_\X a(x-y)u(y,t)\,dy-m u(x,t)-u(x,t) (Gu)(x,t), \\[2mm]
u(x,0)=u_0(x)
\end{cases}
\end{equation}
in the space $E:=L^\infty(\X,dx)$ with the standard $\esssup$-norm.
Here $\ka, m >0$ are constants; $a$ is an (essentially) bounded probability kernel on $\X$, i.e. 
\begin{equation*}
  0\leq a\in L^1(\X,dx)\cap E, \qquad \int_\X a(x)\,dx=1;
\end{equation*}
and $G$ is a nonnegative continuous mapping on $E$ which is acting in $x$, i.e. 
$ (Gu)(x,t):=\bigl(Gu(\cdot,t)\bigr)(x)\geq0$ for $u\geq0$. 
 Here and below, we write $v\leq w$ for $v,w\in E$, if $v(x) \leq w(x)$ for almost all (a.a. in the sequel) $x\in \X$. Moreover, we will often just write $x\in\X$ omitting `for a.a.' before this.

By a solution to \eqref{eq:basicequation} on $\R_+$, we will understand the so-called classical solution, that is a mapping $u:\R_+\to E$ which is continuous in $t\in\R_+$ and continuously differentiable (in the sense of the norm in $E$) in $t\in(0,\infty)$.

We will assume that:
\begin{assum}\label{assum:kappa>m}
        \beta:=\ka-m>0;
\end{assum}
\vspace{-1.35\baselineskip}
\begin{assum}\label{assum:Gpositive}
        \begin{gathered}
        \textsl{there exists $\theta>0$ such that, for each } 0\leq v\leq\theta,\\
                0=G0 \leq Gv\leq G\theta=\beta.
        \end{gathered} 
\end{assum} 

As a result, $u\equiv0$ and $u\equiv\theta$ are stationary solutions to \eqref{eq:basicequation}. 
We will work under assumptions which ensure that
\begin{enumerate}[label=(\roman*)]
  \item there are not constant stationary solutions to \eqref{eq:basicequation} between $0$ and $\theta$; 
  \item $u\equiv0$ is an asymptotically unstable solution to \eqref{eq:basicequation}, whereas $u\equiv\theta$ is an asymptotically stable one;
  \item for a given $u_0\in E$ with $0\leq u_0\leq\theta$, there exists a unique solution to \eqref{eq:basicequation} such that 
    \begin{equation}\label{eq:intube}
        0\leq u(x,t)\leq\theta, \quad x\in\X, \ t>0;
    \end{equation} 
  \item the solution $u$ satisfies the comparison principle (see Section~\ref{sec:assumandresults} for details).
  \end{enumerate}
In particular, the equation \eqref{eq:basicequation} belongs to the class of the so-called monostable equations, see e.g. \cite{BH2002}.

The function $u(x,t)$ may be interpreted as the local density of an evolving in time system of entities which reproduce themselves, compete, and die. The~reproduction appears according to the dispersion, which is realized via the fecundity rate $\ka$ and the density $a$ of a probability dispersion  distribution. The~death may appear due the constant inner mortality $m>0$ within the system, as well as due to the density dependent rate $Gu$, which describes a competition within the system. 

One can also rewrite the equation \eqref{eq:basicequation} in the reaction-diffusion form
\begin{gather}\label{eq:RDequation}
\dfrac{\partial}{\partial t} u(x,t)=(Lu)(x,t)+(Fu)(x,t),\\
\shortintertext{where}\label{eq:Markovgenerator}
(Lu)(x,t):=\ka\int_\R a(x-y)\bigl(u(y,t)-u(x,t)\bigr)\,dy
\end{gather}
is the generator of a nonlocal diffusion in $\X$ (see e.g. \cite{AMRT2010,BCF2011,KMPZ2016,FT2017b}), and the reaction $F$ is given by
\begin{equation}
Fv:=v(\beta -Gv), \quad v\in E.		\label{eq:FthroughG}
\end{equation}
Then, under assumptions \eqref{assum:kappa>m}--\eqref{assum:Gpositive},
\begin{equation}\label{eq:propertiesofF}
F\theta=F0=0\leq Fv \leq \beta v, \quad 0\leq v\leq \theta.
\end{equation}

On the other hand, to rewrite a given equation in the form \eqref{eq:RDequation} to an equation of the form \eqref{eq:basicequation} with a continuous $G$ given by $Gv=\beta-\frac{Fv}{v}$, $v\in E$  (or, at least, a continuous $G$ at $0$ on $\{v\in E: 0\leq v\leq\theta\}$), we will need to require that the reaction $F$ in \eqref{eq:RDequation} is such that 
\begin{equation}\label{eq:nondegenerate}
  \frac{Fv}{v}\to \beta>0\quad \text{as}\quad v\to0+
\end{equation}
(both convergences are in $E$). Because of $F0=0$, we get then that the Fr\'echet derivative of $F$ must be a (strictly positive) constant mapping. In particular, we do not allow the degenerate reaction $F'(0)=0$, see e.g. \cite{AC2016}. Therefore, we consider a sub-class of monostable reaction-diffusion equations \eqref{eq:RDequation}. 

The solution $u$ to the equation \eqref{eq:RDequation} may be interpreted as a density of a species which invades according to a nonlocal diffusion within the space $\X$ meeting a reaction $F$, see e.g.
\cite{Fif1979,Mur2003}.

\subsection{Description of results}

We will say that sets $\{\La(t)\subset \X,\ t>0\}$ describe the propagation of the front for a solution $u=u(x,t)$ to \eqref{eq:basicequation} if, for all (small enough) $\eps>0$, the following convergences hold:
\begin{align}
   &\lim\limits_{t\to\infty}\essinf\limits_{x\in\La(t-\eps t)} u(x,t)=\theta,\label{eq:whatweareproving1}\\
   &\lim\limits_{t\to\infty}\esssup_{{x\in\X\setminus\La(t+\eps t)}} u(x,t)=0.\label{eq:whatweareproving2}
   \end{align}

Informally speaking, for large times, the solution $u$ becomes arbitrary close to $\theta$ inside the set $\La(t)$ and $u$ becomes arbitrary close to $0$ out of this set. The~intermediate zone $\La (t+\eps t)\setminus \La (t-\eps t)$ in \eqref{eq:whatweareproving1}--\eqref{eq:whatweareproving2} is said to be \emph{the front}, 
or \emph{the transition zone}.
It~can also expand as $t\to\infty$; moreover, in the accelerated case considered in the present article, it will be even with necessity, see \cite{HGR2017}.

The propagation \eqref{eq:whatweareproving1}--\eqref{eq:whatweareproving2} of the front, is said to have a constant speed (or just is linear in time), if $\La(t)=t\, \La(1) $. Here and below $tB:=\{tx:x\in B\}$ for a $B\subset\X$.
In contrast, the effect of an infinite speed of propagation, see~\cite{Yag2009,Gar2011,HGR2017}, is called sometimes in literature an acceleration of the propagation, having in mind, for example, that then $\La(t)=\eta(t)\, \La(1) $ with $\frac{\eta(t)}{t}\to\infty$, $t\to\infty$.

In the present paper the propagation will be described by the sets
\begin{equation}\label{eq:defLa}
    \La(t) = \La(t,c):=\bigl\{x\in\X\bigm\vert c(x)\geq e^{-\beta t}\bigr\}, \quad t>0,
 \end{equation}
where the function $c:\X\to(0,\infty)$ will be appropriately chosen below. 

We start the explanation from a more demonstrative radially symmetric case. Namely, let the kernel be a radially symmetric function:
\begin{equation}\label{eq:radkernel}
a (x)=b(|x|), \quad x\in\X;
\end{equation}
and let the initial condition take either of forms:
\begin{alignat}{2}\label{eq:intu0}
 u_0(x)&=q(|x|), && \quad x\in\X\\
\shortintertext{or}
\label{eq:monu0}
   u_0(x)&=\int_{\Delta(x)}q(|y|)dy, &&\quad x\in\X.
\end{alignat}
Here and in the sequel, $|x|$ denotes the Euclidean norm of an $x\in\X$, and 
\begin{equation}
    \Delta(x):=\bigl\{y\in\X: y_j\geq x_j, \ 1\leq j\leq d\bigr\}, \quad x\in\X. \label{eq:bigDelta}
\end{equation}

We suppose also that either of functions $b,q:\R_+\to\R_+$ (or both) in \eqref{eq:radkernel}--\eqref{eq:monu0} have regular heavy tails at $\infty$. Namely, we describe in Definition~\ref{def:super-subexp} below a class $\Et$ of functions $p:\R_+\to\R_+$ such that, in particular, 
\begin{equation*}
  \lim_{s\to\infty} e^{ks}p(s)=\infty, \quad k>0,
\end{equation*}
i.e. each $p\in\Et$ decays slowly than any exponential function.
The class $\Et $ contains any function which is decreasing at $\infty$ to $0$ and is 
asymptotically proportional at $\infty$ to either of
\begin{equation*}
\begin{aligned}
&(\log s)^\la s^{-(d+\delta)},\quad   && s^\nu (\log s)^\mu \exp\bigl(-c(\log s)^{1+\delta}\bigr), \\
&s^\nu (\log s)^\mu\exp\bigl(-c s^\gamma\bigr), \quad   
&& s^\nu (\log s)^\mu \exp\Bigl(-c\frac{s}{(\log s)^{1+\delta}}\Bigr),
\end{aligned}
\end{equation*}
provided that $\delta,c>0$, $\gamma\in(0,1)$, $\mu, \nu\in\R$, and $\la\in\R$ for $d=1$, however, $\la=0$ for $d>1$.

Moreover, to simplify the formulation of our results, we consider firstly
a particular choice of the reaction $F$ in the equation \eqref{eq:RDequation}.
Namely, for a fixed $\theta>0$, we consider the reaction $F$, given, for $u\in E$, by
\begin{equation}\label{eq:partciularF}
  (Fu)(x)=\alpha f\bigl(u(x)\bigr) + (1- \alpha)\,\frac{\beta}{\theta^k}\, u(x)\biggl(\theta - \int_\X a^-(x-y)u(y)\,dy\biggr)^k,
\end{equation}
with $\alpha\in[0,1]$, $f:\R\to\R$, $k\in\N$, $0\leq a^-\in L^1(\R)$, $\int_\R a^-(x)\,dx=1$. 
We will assume that
\begin{equation}\label{eq:coninlocalcase}
\begin{gathered}
        \frac{f(r)}{r} \textsl{ is Lipschitz continuous on } r\in[0,\theta]; \\
        \lim\limits_{r\to 0 +} \frac{f(r)}{r} = \beta;\\ 
        f(0)=f(\theta)=0;\qquad 0<f(r) \leq \beta r, \ \ r\in(0,\theta),
\end{gathered}
\end{equation}
and also that, for some $\varrho>0$,
\begin{equation}\label{eq:sepfromzerobeta}
            \ka a(x) \geq (1-\alpha) k\beta a^-(x) + \varrho\1_{B_\varrho(0)}(x),\qquad x\in\X.
\end{equation}
Note that \eqref{eq:partciularF}--\eqref{eq:coninlocalcase} ensure, in particular, that \eqref{eq:propertiesofF}--\eqref{eq:nondegenerate} hold.
\begin{theorem} \label{thm:radsym}
Let $b,q:\R_+\to\R_+$ be such that, for some \mbox{$M,r,\delta,\rho>0$},
 \begin{equation*}
b(s)+q(s)\leq \frac{M}{(1+s)^{d+1+\delta}} \quad \text{for a.a.~} s \geq r,
 \end{equation*}
and let $q(s)\geq \rho$ for a.a. $s\in[0,\rho]$. 
We assume also that either of the following conditions holds
\begin{align}
&\sup\limits_{s\in\R_+}\dfrac{q(s)}{b(s)}<\infty, \label{eq:alt1}
\\
&\sup\limits_{s\in\R_+}\dfrac{b(s)}{q(s)}<\infty. \label{eq:alt2}
\end{align}
Let $\ka,m >0$ and \eqref{assum:kappa>m} hold. Let $a=a(x)$ be given by \eqref{eq:radkernel}. Consider, for some $\theta>0$, a~reaction $F:E\to E$ given by \eqref{eq:partciularF}, and suppose that \eqref{eq:coninlocalcase}--\eqref{eq:sepfromzerobeta} hold.  
\begin{enumerate}
  \item Let $q:\R\to[0,\theta]$ and let $u_0=u_0(x)$ be given by \eqref{eq:intu0}.
 Then \eqref{eq:whatweareproving1}--\eqref{eq:whatweareproving2} hold with $\La(t)=\La(t,c)$ given by \eqref{eq:defLa}, where
  \begin{enumerate}[label=(\roman*)]
    \item $c= a $, if $b\in\Et $ and \eqref{eq:alt1} holds;
    \item $c=u_0$, if $q\in\Et $ and \eqref{eq:alt2} holds.
  \end{enumerate}
  \item Let $\displaystyle\int_0^\infty q(s) s^{d-1}ds\in(0,\theta]$ and let $u_0=u_0(x)$ be given by \eqref{eq:monu0}. 
  Then \eqref{eq:whatweareproving1}--\eqref{eq:whatweareproving2} hold with $\La(t)=\La(t,c)$ given by \eqref{eq:defLa}, where
  \begin{enumerate}[label=(\roman*)]
    \item $c(x):=\int_{\Delta(x)} a (y)dy$, $x\in\X$, if $b\in\Et $ and \eqref{eq:alt1} holds;
    \item $c=u_0$, if $q\in\Et $ and \eqref{eq:alt2} holds.
  \end{enumerate}
\end{enumerate}

If both \eqref{eq:alt1}--\eqref{eq:alt2} hold, then \eqref{eq:whatweareproving1}--\eqref{eq:whatweareproving2} take place with $\La(t)=\La(t,c)$ given by \eqref{eq:defLa} and constructed by either of the corresponding functions $c=c(x)$.
\end{theorem} 

Informally speaking, $c$ in \eqref{eq:whatweareproving1}--\eqref{eq:defLa} is  
\begin{center}
  either $a$ or $u_0$, whichever decays slower, if \eqref{eq:intu0} holds;\\[2mm]
  either $\int_{\Delta(x)}a$ or $u_0$, whichever decays slower, if \eqref{eq:monu0} holds.
\end{center} 

On the sketches below, we shade the sets $\La(t-\eps t)$ and $\X\setminus\La(t+\eps t)$ for both types of initial conditions \eqref{eq:intu0}--\eqref{eq:monu0}.
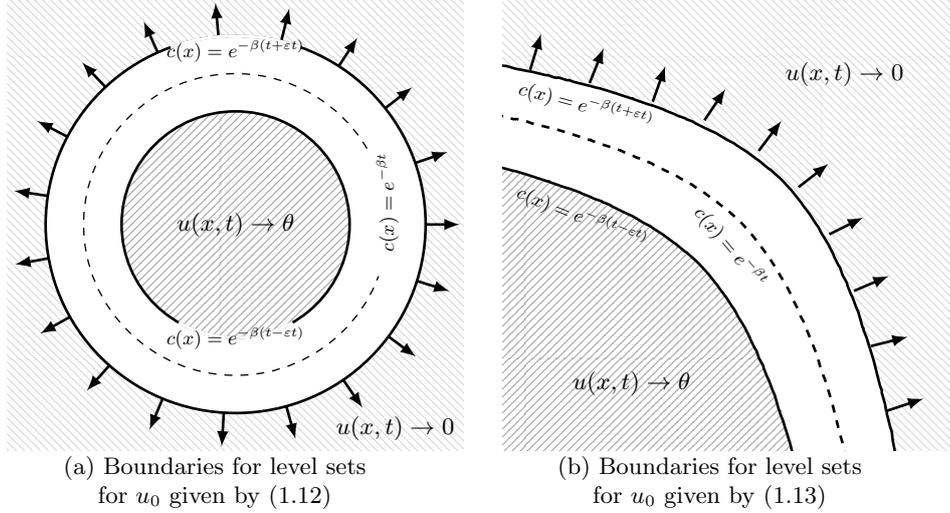
\begin{figure}[!htbp]
\vspace{-0.75\baselineskip}
\centering
\begin{subfigure}[t]{.45\linewidth}
  \centering 
    \begin{tikzpicture}[line width=1pt,every node/.style={font=\small},>=stealth]
        \fill[pattern=north west lines, pattern color=gray!30!white] (-3,-3) rectangle (3,3);
        \fill[white] (0,0) circle (2.5);
        \draw (0,0) circle (2.5);
        \draw[decorate,decoration={arrowborder,amplitude=0.45cm,angle=90,segment length=9.8mm}] (0,0) circle (2.95);
        \draw[dashed, line width=0.5pt] (0,0) circle (2);
        \fill[pattern=north east lines, pattern color=gray!70!white] (0,0) circle (1.5);
        \draw (0,0) circle (1.5);
        \draw[line width=1.2pt,white] ({2.5*cos(70)},{2.5*sin(70)}) arc (70:110:2.5);
        \draw[line width=1.2pt,white] ({1.5*cos(240)},{1.5*sin(240)}) arc (240:300:1.5);
        \draw[line width=1.2pt,white] ({2*cos(-20)},{2*sin(-20)}) arc (-20:30:2);
        \node at (2,0.3) {\rotatebox{90}{\scalebox{0.8}{$c(x)=e^{-\beta t}$}}};
        \node at (0,-1.5) {\scalebox{0.8}{$c(x)=e^{-\beta (t-\eps t)}$}};
        \node at (0,2.3) {\scalebox{0.8}{$c(x)=e^{-\beta (t+\eps t)}$}};
        \node at (0,0) {$u(x,t)\to\theta$};
        \node at (2.1,-2.7) {$u(x,t)\to0$};
    \end{tikzpicture}
    \caption{Boundaries for level sets\\ for $u_0$ given~by~\eqref{eq:intu0}}
\end{subfigure}\qquad\quad
\begin{subfigure}[t]{.45\linewidth}
\centering  
\begin{tikzpicture}[line width=1pt,every node/.style={font=\small},>=stealth]
\newcommand\cpath{(2.97, -2.26) -- (2.97, -2.22) -- (2.94, -2.07) -- (2.93, -2.01) -- (2.91, -1.91) -- (2.90, -1.86) -- (2.88, -1.76) -- (2.88, -1.75) -- (2.86, -1.65) -- (2.85, -1.60) -- (2.84, -1.50) -- (2.82, -1.44) -- (2.78, -1.28) -- (2.78, -1.25) -- (2.77, -1.24) -- (2.76, -1.17) -- (2.72, -1.03) -- (2.72, -1.00) -- (2.70, -0.95) -- (2.67, -0.83) -- (2.65, -0.75) -- (2.62, -0.63) -- (2.61, -0.60) -- (2.58, -0.50) -- (2.56, -0.44) -- (2.54, -0.37) -- (2.50, -0.27) -- (2.50, -0.25) -- (2.50, -0.25) -- (2.49, -0.24) -- (2.44, -0.06) -- (2.41, 0.01) -- (2.37, 0.12) -- (2.35, 0.16) -- (2.31, 0.26) -- (2.30, 0.30) -- (2.25, 0.39) -- (2.22, 0.48) -- (2.20, 0.51) -- (2.15, 0.61) -- (2.14, 0.64) -- (2.08, 0.76) -- (2.05, 0.81) -- (2.00, 0.89) -- (1.96, 0.97) -- (1.93, 1.01) -- (1.86, 1.12) -- (1.77, 1.24) -- (1.76, 1.26) -- (1.75, 1.27) -- (1.75, 1.27) -- (1.64, 1.41) -- (1.54, 1.52) -- (1.52, 1.54) -- (1.50, 1.56) -- (1.41, 1.64) -- (1.39, 1.66) -- (1.37, 1.67) -- (1.25, 1.77) -- (1.25, 1.77) -- (1.25, 1.77) -- (1.22, 1.79) -- (1.10, 1.88) -- (0.99, 1.95) -- (0.95, 1.98) -- (0.87, 2.02) -- (0.74, 2.10) -- (0.63, 2.15) -- (0.59, 2.17) -- (0.49, 2.22) -- (0.46, 2.24) -- (0.38, 2.27) -- (0.37, 2.28) -- (0.28, 2.31) -- (0.24, 2.33) -- (0.14, 2.37) -- (0.10, 2.39) -- (-0.01, 2.43) -- (-0.08, 2.45) -- (-0.25, 2.51) -- (-0.26, 2.52) -- (-0.27, 2.52) -- (-0.28, 2.52) -- (-0.46, 2.58) -- (-0.52, 2.60) -- (-0.62, 2.63) -- (-0.65, 2.64) -- (-0.70, 2.65) -- (-0.77, 2.67) -- (-0.85, 2.69) -- (-0.97, 2.72) -- (-1.02, 2.74) -- (-1.05, 2.74) -- (-1.14, 2.77) -- (-1.19, 2.77) -- (-1.25, 2.79) -- (-1.27, 2.80) -- (-1.30, 2.80) -- (-1.46, 2.84) -- (-1.52, 2.85) -- (-1.62, 2.87) -- (-1.67, 2.88) -- (-1.77, 2.91) -- (-1.88, 2.92) -- (-1.93, 2.93) -- (-2.03, 2.95) -- (-2.09, 2.96) -- (-2.24, 2.99) -- (-2.28, 2.99)}
   \begin{scope}   
        \clip (-2,-1.9) rectangle (4,4.1);  
        \begin{scope}[scale=1.1]
         \fill[pattern=north west lines, pattern color=gray!30!white] (-3,-3) rectangle (4,4);
        \fill[white] \cpath -- (-3,-3) -- cycle;
        \draw \cpath; 
        \draw[decorate,decoration={arrowborder,amplitude=0.5cm,angle=90,segment length=8.5mm},xshift=0.38cm,yshift=0.4cm] \cpath; 
        \fill[white,xshift=-0.5cm,yshift=-0.5cm] \cpath -- (-3,-3) -- cycle;
        \draw[xshift=-0.5cm,yshift=-0.5cm,dashed] \cpath;  
        \fill[pattern=north east lines, pattern color=gray!70!white,xshift=-1cm,yshift=-1cm] \cpath -- (-3,-3) -- cycle;
        \draw[xshift=-1cm,yshift=-1cm] \cpath;
         \end{scope} 
        \node at (1.0,0.8) {\rotatebox{-50}{\scalebox{0.8}{$c(x)=e^{-\beta t}$}}};
        \node at (-0.95,1.2) {\rotatebox{-22}{\scalebox{0.8}{$c(x)=e^{-\beta (t-\eps t)}$}}};
        \node at (-0.9,2.68) {\rotatebox{-15}{\scalebox{0.8}{$c(x)=e^{-\beta (t+\eps t)}$}}};
        \node at (-0.3,-1.0) {$u(x,t)\to\theta$};
        \node at (2.5,3.1) {$u(x,t)\to0$};
    \end{scope}
    \end{tikzpicture}
      \caption{Boundaries for level sets\\ for $u_0$ given~by~\eqref{eq:monu0}}
\end{subfigure}

\vspace{-0.75\baselineskip}
\caption{Boundaries for level sets $\La(t)$, $\La(t-\eps t)$, $\La(t+\eps t)$\\ for two classes of~initial~conditions~$u_0$ (in $\R^2$);\\ the arrows show the directions of the propagation}\label{fig:frontsimple}
\end{figure} 

In Section~\ref{sec:assumandresults} below, we generalize Theorem~\ref{thm:radsym} by considering, instead of \eqref{eq:partciularF}, a class of reactions in \eqref{eq:RDequation} which satisfy \eqref{eq:propertiesofF}--\eqref{eq:nondegenerate}. Moreover, we weaken the assumption that the kernel and the initial condition in \eqref{eq:radkernel}--\eqref{eq:monu0} are constructed by radially symmetric functions, by allowing that each of $a=a(x)$ and $u_0=u_0(x)$ may fluctuate in an appropriate way. To demonstrate possible fluctuations, we present the  examples below, which correspond to different growths of $\eta(t)$ in $\La(t)=\eta(t)\La(1)$ for \eqref{eq:whatweareproving1}--\eqref{eq:whatweareproving2}.

\begin{example}
We start with the case when $u_0\in L^1(\X)$, cf.~\eqref{eq:intu0}. Then $c$ in \eqref{eq:whatweareproving1}--\eqref{eq:defLa} will be chosen in the form $c(x)=b(|x|)$, $x\in\X$, with a decreasing at $\infty$ function $b:\R_+\to\R_+$. In this case, the set \eqref{eq:defLa} is given by
\begin{equation*}\label{eq:explicitLa}
  \La(t,c)=\bigl\{x\in\X \bigm\vert |x|\leq \eta(t)\}, \qquad \eta(t):=b^{-1}\bigl(e^{-\beta t}\bigr)
\end{equation*}
for large enough $t$ (to invert $b$). We assume that $b_{\max{}},b_{\min{}}:\R_+\to\R_+$ and $r>0$ are such that the following two conditions hold:
\begin{gather*}
\max\{a (x),u_0(x)\}\leq b_{\max{}}(|x|), \ \ |x|\geq r;\\
\text{either} \ \  a (x)\geq b_{\min{}}(|x|), \ \ |x|\geq r, \qquad \text{or}\ \ 
u_0 (x)\geq b_{\min{}}(|x|), \ \ |x|\geq r.
\end{gather*}
Then \eqref{eq:whatweareproving1}--\eqref{eq:whatweareproving2} hold with $\La(t)=\La(t,c)$ given by \eqref{eq:explicitLa}, where $\eta(t)$ can be found from the following table:
\begin{center}
\begin{tabular}{|P{0.35\textwidth}|P{0.35\textwidth}|P{0.17\textwidth}|N} 
\hline
$b_{\min{}}(|x|)$ & $b_{\max{}}(|x|)$ & $\eta(t)$ & \\[12pt]\hline
$\dfrac{1}{(\log |x|)^{\nu}} \dfrac{1}{|x|^{d+\mu}}$ & $(\log |x|)^{\nu} \dfrac{1}{|x|^{d+\mu}}$ & $\exp\Bigl(\dfrac{\beta t}{d+\mu}\Bigr)$& \\[23pt]\hline
$\dfrac{1}{|x|^{\nu}} \exp\bigl(-(\log |x|)^{\la}\bigr)$ & $|x|^\nu \exp\bigl(-(\log |x|)^{\la}\bigr)$ & $\exp\bigl((\beta t)^{\frac{1}{\la}}\bigr) $ & \\[23pt]\hline
$\dfrac{1}{|x|^{\nu}} \exp(-|x|^\gamma)$ & $|x|^{\nu} \exp(-|x|^\gamma)$ & 
$(\beta t)^{\frac{1}{\gamma}}$ & \\[23pt]\hline
$\dfrac{1}{|x|^{\nu}} \exp\Bigl(-\dfrac{|x|}{(\log |x|)^\la}\Bigr)$ & $|x|^\nu \exp\Bigl(-\dfrac{|x|}{(\log |x|)^\la}\Bigr)$ & $\substack{\displaystyle\sim \beta t (\log t)^\la\\[1mm]\displaystyle t\to\infty}$ & \\[23pt]\hline
\end{tabular} 
\end{center}
Here $\nu\geq 0$, $\mu>0$, $\la>1$, $\gamma\in(0,1)$. For the first three cases, the calculation of $\eta(t)$ is straightforward; in the last case, the asymptotic of $\eta(t)$ is shown in the Appendix, Lemma~\ref{le:almostlinear}.
\end{example}

In the following two examples, we consider the case of a non-integrable $u_0$, cf.~\eqref{eq:monu0}. 

\begin{example}
Let $d=1$. Then  \eqref{eq:whatweareproving1}--\eqref{eq:defLa} hold with
\[
  \La(t)= [\eta(t),\infty), \qquad \eta(t):=c^{-1}\bigl(e^{-\beta t}\bigr), \qquad c(x)=\int_x^\infty b(y)dy
\]
for large $t$ and $x$, and for a decreasing at $\infty$ function $b\in L^1(\R_+\to\R_+)$.

For example, let, for some $\mu>0$ and $\nu\geq0$,
\begin{equation*}
\dfrac{1}{(\log |x|)^{\nu}} \dfrac{1}{|x|^{1+\mu}}\leq a(x) \leq
  (\log |x|)^{\nu} \dfrac{1}{|x|^{1+\mu}}
\end{equation*}
for large $|x|$; and suppose, for simplicity, that $u_0$ is monotone on $\R$ and, for some $\zeta\in(0,\theta)$,
\[
  \zeta\1_{\R_-}(x)\leq u_0(x)\leq \zeta \int_x^\infty  a (y)dy\leq \theta, \quad x\in\R. 
\]
Then one can choose $b(s)=s^{-1-\mu}$ (for large $s$), and then, for large $t$, 
\begin{equation*}
  \eta(t)=\exp\Bigl(\frac{\beta t}{\mu}\Bigr),
\end{equation*}
i.e. the front propagates a bit faster than in the case of $u_0\in L^1(\R)$.
\end{example}

\begin{example}
Let $d=2$. Then the boundary of $\La(t)$ is described, cf.~Figure~\ref{fig:frontsimple}, for each direction $(\xi_1,\xi_2)\in S^1$ (the unit circle in $\R^2$), by two functions, $X_1(t)=X_1(t,\xi_1,\xi_2)$ and $X_2(t)=X_2(t,\xi_1,\xi_2)$, such that
\begin{equation}\label{eq:frontinR2}
\int_{X_1(t)}^\infty\int_{X_2(t)}^\infty b(|y|)\,dy_1\,dy_2=e^{-\beta t}, \qquad |y|=\sqrt{y_1^2+y_2^2},
\end{equation}
for a decreasing at $\infty$ function $b:\R_+\to\R_+$, such that $\int_{\R_+} b(r)r\,dr<\infty$. Consider the acceleration of the front propagation along the diagonal direction $\xi_1=\xi_2=\frac{1}{\sqrt{2}}$; we set then $X(t):=X_1(t)=X_2(t)$.

Let, for example, for $\nu\geq0$,
\begin{equation}\label{eq:weibexmon1}
\dfrac{1}{|x|^{\nu}} \exp(-\sqrt{|x|})\leq a(x) \leq |x|^\nu  \exp(-\sqrt{|x|}),
\end{equation}
for large $|x|$; and suppose that, for some $\zeta\in(0,\theta)$ and a.a. $(x_1,x_2)\in\R^2$,
\begin{equation}\label{eq:weibexmon2}
  \zeta\1_{\R_-^2}(x_1,x_2)\leq u_0(x_1,x_2)\leq \zeta \int_{x_1}^\infty\int_{x_2}^\infty  a (y_1,y_2)\,dy_1\,dy_2\leq \theta.
\end{equation}
Then, for each $\eps\in(0,1)$ and for all large enough $t$,
\begin{equation}\label{eq:weibestforeta}
  \frac{1}{2\sqrt{2}}\beta^2(1-\eps)^2t^2\leq X(t)\leq \frac{1}{\sqrt{2}}\beta^2t^2,
\end{equation}
see Lemma~\ref{le:2dim} and Remark~\ref{rem:forweibR2ex} in the Appendix for details.
\end{example}

\subsection{Overview of literature}
In the recent decades, there is a growing interest to the study of nonlocal monostable reaction-diffusion equations: for a pure local reaction \eqref{eq:partciularF} with $\alpha=1$, see e.g. \cite{BCV2016,BCGR2014,CDM2008,CD2005,Yag2009,Gar2011,AGT2012,SLW2011,SZ2010,XLR2018}; for a pure nonlocal reaction \eqref{eq:partciularF} with $\alpha=0$ and $k=1$, see \cite{Dur1988,PS2005,FM2004,FKK2011a,FKT100-1,FKT100-2,FKT100-3}; and for the origins of the topic, see also \cite{Sch1980,Die1978a,Aro1977,Wei1978,Mol1972a,Mol1972}. 

Two classical examples which satisfy \eqref{eq:coninlocalcase} are $f(s)=\nu s(\theta-s)$, cf. \cite{Fis1937}, and $f(s)=\nu s(\theta-s)^2$, cf. \cite{KPP1937}; for some $\nu>0$. Note that, if $f$ in \eqref{eq:coninlocalcase} is differentiable at $0$, then we require $f'(0)=\beta>0$ and $f(u)\leq f'(0)u$ for all $0\leq u\leq \theta$. 
The importance of the latter assumption 
for the front propagation was pointed out in e.g. \cite{CD2005,BCGR2014}, it leads to the possibility to describe the front using the linearized version of the corresponding equations \eqref{eq:basicequation} and \eqref{eq:RDequation} about $0$, that is just \eqref{eq:majorequation} below; note also that this assumption can be weaken, see \cite{Wei2012}. The degenerate case $f'(0)=0$ was considered in e.g. \cite{ZLW2012,AC2016}.

Now we discuss the existing results for both linear in time and accelerated propagations.

Mollison \cite{Mol1972,Mol1972a} studied, for the dimension $d=1$, a local reaction $F$ given by \eqref{eq:partciularF} with $\alpha=0$, $k=1$, $a^-=a$, and $\beta=\theta=\ka$ (that corresponds to $m=0$), for a monotone initial condition $u_0$, cf.~\eqref{eq:monu0}. He has shown that
 the~property of the corresponding propagation front to have an `averaged' constant speed is necessary and sufficient with the existence of a $\la>0$, such that
\[
  \int_\R a(s)e^{\la s}ds<\infty,  \qquad \sup_{s\in\R}u_0(s)e^{\la s}<\infty.
\]
Note that such $u_0$ gives an unbounded set $\La(t)=(-\infty,\ga t)$ 
 in \eqref{eq:whatweareproving1}--\eqref{eq:whatweareproving2}.

For $d\geq1$, we have shown in \cite[Proposition 3.1]{FKT100-3} that similar restrictions 
\begin{equation}\label{eq:lighttails}
\exists\,\mu>0:\ \int_\X a(x)e^{\mu |x|}dx<\infty, \qquad \forall\,\la>0:\ \esssup_{x\in\X}u_0(x)e^{\la |x|}<\infty
\end{equation}
yield that the solution $u$ to the equation \eqref{eq:basicequation} propagates at most linearly in any direction $\xi\in\X$. Note that \eqref{eq:lighttails} implies $u_0\in L^1(\X)$, cf.~\eqref{eq:intu0}. Moreover, for the reaction $F$ given by \eqref{eq:partciularF} with $k=1$ and under \eqref{eq:sepfromzerobeta},
we have proved in \cite[Propositions~4.7 and 4.2]{FKT100-3}  
that the assumption \eqref{eq:lighttails} implies that the convergences \eqref{eq:whatweareproving1}--\eqref{eq:whatweareproving2} hold with $\La(t)=t\La(1)$, where $\La(1)$ is a bounded convex subset of $\X$. 
For the particular case of $\alpha=0$, $k=1$, $a=a^-$ in \eqref{eq:partciularF}, a similar result was obtained in \cite{PS2005}.

The conditions \eqref{eq:lighttails} are closed to the necessary ones, cf.~\cite{Yag2009,Gar2011}.
We have proved in \cite[Proposition 1.4]{FKT100-3} (cf. also \cite{Gar2011}) that, if a bit weaker form of \eqref{eq:lighttails} fails for $a$ (roughly, if $a$ is `heavier' than any exponent at infinity), then the convergence \eqref{eq:whatweareproving1} holds with $\La(t)=tK$ for an {\em arbitrary} compact set $K\subset\X$.
Therefore, the propagation of the front is faster than linear.

For the dimension $d=1$ and for the local reaction $F$ given by \eqref{eq:partciularF} with $\alpha=1$, the acceleration was known in mathematical biology, see e.g. the results and references in \cite{MK2003,HF2014,LPP2016}.
The first rigorous result in this direction was done by Garnier~\cite{Gar2011}, who proved an analogue of \eqref{eq:whatweareproving1}--\eqref{eq:whatweareproving2} for $d=1$ and a compactly supported initial condition $u_0$.
However, in his approach, the set $\La\bigl(t+\eps t,c\bigr)$ in \eqref{eq:whatweareproving2} given by \eqref{eq:defLa} was replaced by $\La\bigl(\ga t,c\bigr)$ with some (unknown) $\ga>1$, i.e. the result was not sharp. Note that the technique in \cite{Gar2011} was inspired by \cite{HR2010}, where an acceleration was shown for the  classical KPP-equation with a slowly decreasing initial condition; see also a recent paper \cite{Hen2016}.

Further progress in the study of the acceleration in the dimension $d=1$, for a local reaction $F$ (with $\alpha=1$ in \eqref{eq:partciularF}), was done recently. In \cite{BGHP2017}, both $a$ and $u_0$ are supposed to be symmetric, with a heavy-tailed $a$; the technique used there goes back to \cite{ES1989}.  In \cite{AC2016}, the case of $f'(0)=0$ was considered (that does not covered by the present paper, because of \eqref{eq:nondegenerate}), then $a$ does not need to be symmetric, and $u_0$ is separated from $0$ at $-\infty$, however, $u_0(x)=0$ for large $x$. 

In \cite{FT2017c}, we considered, for the case $d=1$, a general reaction $F$ which satisfies \eqref{eq:propertiesofF}--\eqref{eq:nondegenerate} such that the corresponding $G$ fulfills the assumptions of Section~\ref{sec:assumandresults} below. In this case, the result of Theorem~\ref{thm:radsym} was extended to functions $a$ and $u_0$ which have different orders of decreasing at $\pm\infty$; for example, $a(x)=\exp(-x^\gamma)$, $x>r$, and $a(x)=(1-x)^{-1-\delta}$, $x<-r$, for some $\gamma\in(0,1)$, $\delta,r>0$.

Therefore, up to our knowledge, the present paper is the first one which contains results about the acceleration in \eqref{eq:whatweareproving1}--\eqref{eq:whatweareproving2} for the multidimensional case $d>1$.
 
Note also that analogous results were obtained in \cite{CR2013,CR2012} for the equation of the type \eqref{eq:RDequation} with a local reaction $F$, where $L$ in \eqref{eq:Markovgenerator}, was replaced by a~fractional Laplacian (in particular, the kernel $a$ was singular and non-integrable); cf.~also \cite{FY2013,MM2015,CCR2012}.

The present paper is organized as follows.
In~Subsection~\ref{subsec:assum}, we formulate further assumptions on $a$ and $G$ and known results about solutions to \eqref{eq:basicequation}. Note that these assumptions are fulfilled for the reaction \eqref{eq:partciularF} in conditions of Theorem~\ref{thm:radsym}, see Lemma~\ref{le:GisOK}. In Subsection~\ref{subsec:statements}, we describe the mentioned class $\Et$ of regular heavy-tailed functions and formulate the main result, Theorem~\ref{thm:combined}, which generalizes Theorem~\ref{thm:radsym} (cf.~Remark~\ref{rem:radsym}). In Section~\ref{sec:stratofproof}, we present a scheme of the proof for Theorem~\ref{thm:combined}. Section~\ref{sec:techtools} contains technical tools, mainly about the properties of sets $\La(t)$ in \eqref{eq:whatweareproving1}--\eqref{eq:defLa}. In~Section~\ref{sec:proof}, we present a detailed proof of Theorem~\ref{thm:combined}. Finally, the Appendix contains, in particular, the proof of the mentioned Lemma~\ref{le:GisOK}.

\section{Assumptions and general results}\label{sec:assumandresults}
\subsection{Further assumptions}\label{subsec:assum}
Consider further assumptions to \eqref{assum:kappa>m}--\eqref{assum:Gpositive}.
Set, for an $r>0$,
\[
	E^+:=\{v\in E\mid v\geq0\}, \qquad E_r^+:=\{v\in E\mid 0 \leq v \leq r\}.
\] 

We assume that $G$ is (locally) Lipschitz continuous in $E_\theta^+$, namely, 
\begin{assum}\label{assum:Glipschitz}
	\begin{gathered}
		\textsl{there exists $l_\theta>0$ such that}\\
		\|Gv-Gw\| \leq l_\theta\|v-w\|,\quad v,w \in E_\theta^+.
	\end{gathered}
\end{assum}

We restrict ourselves to the case when the comparison principle for \eqref{eq:basicequation} holds.
Namely, we assume that the right hand side (r.h.s. in the sequel) of \eqref{eq:basicequation} is a~(quasi-) monotone operator:
\begin{assum}\label{assum:sufficient_for_comparison}
	\begin{gathered}
		\textsl{for some $p\geq0$ and for any $v,w\in E_\theta^+$ with $v\leq w$},\\
  	\ka a*v -v\, Gv + pv \leq \ka a*w -w\, Gw + pw.
	\end{gathered}
\end{assum}
Here and below $*$ means the classical convolution over $\X$, i.e., for $u\in E$,
\begin{equation}\label{eq:defconv}
  (a*u)(x):=\int_\X a(x-y)u(y)\,dx, \quad x\in\X.
\end{equation}

\begin{theorem}[\!\!{\cite[Theorems 2.2, 2.3, Proposition 4.2]{FT2017a}}]\label{thm:existandcompared}
Let \eqref{assum:kappa>m}--\eqref{assum:sufficient_for_comparison} hold. Let $u_0\in E^+_\theta$. Then, for each $T>0$, there exists a unique solution $u=u(x,t)$ to \eqref{eq:basicequation} for $t\in[0,T]$; and \eqref{eq:intube} holds for all $t>0$. Moreover, if $v_0\in E^+_\theta$, $v=v(x,t)$ is the corresponding solution to $\eqref{eq:basicequation}$, and if $u_0\leq v_0$, then $0\leq u(\cdot,t)\leq v(\cdot,t)\leq \theta$ for all $t>0$.
\end{theorem}

Note also that if $u_0$ is (uniformly) continuous function on $\X$, then $u(\cdot,t)$ will be also (uniformly) continuous function on $\X$ for all $t>0$. The comparison between solutions in Theorem~\ref{thm:existandcompared} was a part of a more general result, which we will also use.

\begin{theorem}[\!\!{\cite[Theorems 2.2]{FT2017a}}]\label{thm:comparison}
	Let \eqref{assum:kappa>m}--\eqref{assum:sufficient_for_comparison} hold.  Let $T>0$ be fixed. Suppose that  $u_1,u_2:[0,T]\to E$ are continuous mappings, continuously differentiable in $t\in(0,T]$, and such that, for $(x,t)\in\X \times(0,T]$,
    \begin{gather*}
		\frac{\partial u_1}{\partial t} - \ka a*u_1 +mu_1 +u_1Gu_1 \leq \frac{\partial u_2}{\partial t} - \ka a*u_2 +mu_2 +u_2Gu_2,\\
u_1(x,t)\geq0, \qquad u_2(x,t)\leq \theta,\qquad
      0\leq u_{1}(x,0)\leq u_{2}(x,0)\leq \theta.
    \end{gather*}
  	Then $u_{1}(x,t)\leq u_{2}(x,t)$ for $(x,t)\in\X \times[0,T]$. 
\end{theorem}
Note that here and in the sequel, with an abuse of notations, the symbol $*$ stands for the convolution in $x$ variable only, when $u=u(x,t)$ in \eqref{eq:defconv}.

We assume also that the kernel $a$ is not degenerate at the origin, namely,
\begin{assum}\label{assum:a_nodeg}
  \textsl{there exists $\rho>0$ such that} \ a(x)\geq\rho \text{ for a.a. } x\in B_\rho(0).
\end{assum}

Consider on $E$ the topology of the locally uniform convergence: a sequence
 $(v_n)_{n\in\N}\subset E$ is said to be convergent to $v \in E$ locally uniformly, which we denote $v_n \locun v$, $n\to \infty$, if, for each compact set $\La\subset\X$, the sequence $(\1_\La v_n)_{n\in\N}$ converges to $\1_\La v$ in $E$ as $n\to \infty$; here and below $\1_B$ denotes the indicator-function of a $B\subset\X$. 

Stability of the solution to \eqref{eq:basicequation} with respect to the initial condition in the topology of the locally uniform convergence requires continuity of $G$ in this topology. Namely, we assume that
\begin{assum}\label{assum:G_locally_continuous}
	\begin{gathered}	
		\textsl{for any $v_n, v \in E_\theta^+$ such that $v_n \locun v$, $n \to \infty$, one has}\\
		Gv_n \locun Gv, \ n \to \infty.
	\end{gathered}
\end{assum}

We will consider the translation invariant case only. Namely, for each $y\in\X$, we denote by $T_y:E\to E$ the translation operator, that is $(T_y v)(x):=v(x-y)$, $x\in\X$. Then we assume that
\begin{assum}\label{assum:G_commute_T}
		(T_y G v)(x) = (G T_y v)(x),\quad v \in E_\theta^+,\  x\in\X.
\end{assum}

Under \eqref{assum:G_commute_T}, for any $r\equiv const\in (0,\theta)$, $Gr\equiv const$. In this case, we assume also that
\begin{assum}\label{assum:G_increas_on_const}
  Gr < \beta, \qquad r\in(0,\theta).	
\end{assum}

Finally, we will distinguish two cases. If the condition
\begin{assum}\label{assum:first_moment_finite}
	\int_\X \lvert y\rvert a(y)dy<\infty
\end{assum}
holds, then we assume, additionally to \eqref{assum:sufficient_for_comparison}, that
\begin{assum}\label{assum:improved_sufficient_for_comparison}
	\begin{gathered}
		\textsl{there exist $p\geq 0$, $b\in C^\infty(\X)\cap L^\infty(\X)$, $\delta>0$ such that}\\ 
		a(x)-b(x)\geq \delta\1_{B_\delta(0)}(x), \quad x\in\X,\\
  	w\, Gw\leq \ka b*w + pw, \quad w\in E_\theta^+,
	\end{gathered}
\end{assum}
and also
\begin{assum}\label{firstfullmoment}
  \int_\X x a(x)\,dx=0\in\X.
\end{assum}
Otherwise, if \eqref{assum:first_moment_finite} does not hold, then we assume that,
\begin{assum}\label{assum:approx_of_basic}
 \begin{gathered}
    \textsl{let \eqref{assum:kappa>m}--\eqref{assum:sufficient_for_comparison} hold, and for each $n\in\N$, there exist}\\
    0\leq a_n\in L^1(\X), \quad \ka_n>0, \quad G_n:E\to E, \quad \theta_n\in(0,\theta]\\ 
    \textsl{which satisfy \eqref{assum:kappa>m}--\eqref{firstfullmoment} instead of $a$, $\ka$, $G$, $\theta$, respectively,}\\
    \textsl{such that } \theta_n\to\theta, \ n\to\infty, \ \text{and}\\ 
    \ka_n a_n*w -wG_nw \leq \ka a*w - wGw,\quad w\in E_{\theta_n}^+, \ n\in\N.
  \end{gathered} 
\end{assum}

The following statement describes the so-called \emph{hair-trigger} effect. 
\begin{theorem}[cf. {\cite[Theorems 2.5, 2.7]{FT2017a}}]\label{thm:hairtrigger}
	Let either \eqref{assum:kappa>m}--\eqref{firstfullmoment} hold or \eqref{assum:approx_of_basic} hold.
Let $u_0\in E_\theta^+$, $u_0\not\equiv0$, and let $u=u(x,t)$ be the corresponding solution to~\eqref{eq:basicequation}.
Then, for each compact set $K\subset \X$,
  \begin{equation*}
    \lim_{t\to\infty} \essinf_{x\in K} u(x,t)=\theta.
  \end{equation*}
\end{theorem}
\begin{remark}\label{rem:agreementnonzero}
For a brevity of notations, we will treat $u_0\not\equiv0$ as follows: 
there exist $\rho>0$ and $x_0\in\X$ such that $u_0(x)\geq \rho$ for a.a.~$x\in B_\rho(x_0)$. 
\end{remark}

The following lemma shows that the mapping $G$ corresponding to the reaction \eqref{eq:partciularF} satisfies the assumptions above. Its proof is provided in the Appendix.
\begin{lemma}\label{le:GisOK}
Let \eqref{assum:kappa>m}, \eqref{assum:first_moment_finite} and \eqref{firstfullmoment} hold. Let $F$ be given by \eqref{eq:partciularF}.
Suppose that, for some $\varrho>0$, \eqref{eq:sepfromzerobeta} holds, and, if $\alpha\neq 0$, suppose also that 
\eqref{eq:coninlocalcase} holds. Set, for a.a. $x\in\X$,
\begin{equation}\label{eq:defGbyF}
(Gu)(x):=\beta-\frac{(Fu)(x)}{u(x)}, \ \  u(x)\neq0, \qquad (Gu)(x):=0, \ \ u(x)=0.
\end{equation}
Then \eqref{assum:Gpositive}--\eqref{assum:G_increas_on_const} and \eqref{assum:improved_sufficient_for_comparison} hold.
\end{lemma}

\subsection{General results}\label{subsec:statements}
In this Subsection, we present a general result, Theorem~\ref{thm:combined}, which generalizes Theorem~\ref{thm:radsym} to a non radially symmetric case and to the case of a reaction $F$ in \eqref{eq:RDequation} which corresponds to a mapping $G:E\to E$ in \eqref{eq:basicequation} which fulfills the assumptions above. We start with a description of the class $\Et $ of functions on $\R_+$ which is used in the both Theorems.

\begin{definition}\label{def:right-sideclasses}
 Let a measurable function $b:\R_+\to\R_+$ be such that $b(s)>0$ for all $s\geq\rho$ with some $\rho=\rho_b\geq0$. Then the function $b$ is said to be (right-side)
 \begin{description}
 	\item[--] {\em long-tailed}, if, for any $\tau\geq 0$,
                \begin{equation}\label{eq:longtaileddef}
                        \lim_{s\to\infty}\frac{b(s+\tau)}{b(s)}=1;
                \end{equation}
   \item[--] {\em tail-decreasing}, if $b=b(s)$ is strictly decreasing on $[\rho,\infty)$ to $0$;
   \item[--] {\em tail-log-convex}, if the function $\log b$ is convex (and hence continuous) on $(\rho,\infty)$.
 \end{description}
\end{definition}

\begin{definition}\label{def:super-subexp}
Let $d\in\N$.
  \begin{enumerate}[label={\arabic*})]
    \item Let $\Dt$ be the set of all bounded tail-decreasing functions $b:\R_+\to(0,\infty)$ such that $\int_0^\infty b(s) s^{d-1}\,ds<\infty$ and $\inf\limits_{|s|<r}b(s)>0$ for each $r>0$;
    \item Let $\Et \subset\Dt$ be the set of all functions $b\in\Dt$ such that  
   \begin{description}
    \item[--] $b$ is long-tailed and tail-log-convex, and $b(\rho)\leq 1$  (without loss of generality);
  \item[--]  there exist $\delta=\delta_b\in(0,1)$ and an increasing function $h=h_b:(0,\infty)\to(0,\infty)$, with
$h(s)<\dfrac{s}{2}$ and $\lim\limits_{s\to\infty}h(s)=\infty$ such that
\[
\lim_{s\to\infty}\frac{b(s\pm h(s))}{b(s)}=1, \qquad
	\lim_{s\to\infty} b\bigl(h(s)\bigr)s^{1+\delta}=0;
\]
\item[--] and, for the case $d>1$, we assume, additionally, that \emph{either}, for some $\mu,M>0$,
\begin{equation}\label{eq:polynomialb}
b(s)=\frac{M}{(1+s)^{d+\mu }}, \quad s\in\R_+,
\end{equation}
\emph{or}, for all $\nu\geq1$,
\begin{equation}\label{eq:quicklydecreasing}
\lim_{s\to\infty} b(s)s^\nu =0.
\end{equation}
\end{description}
\end{enumerate}
\end{definition}

The peculiarities of functions from $\Et $ can be found in \cite{FT2017b}.

As it was mentioned before, we are going to weaken the radially symmetric conditions of Theorem~\ref{thm:radsym}. Namely, we assume that there exist bounded measurable functions $b^+,b_+:\R_+\to\R_+$ and $v^\circ, v_\circ:\X\to[0,\theta]$ such that
\begin{alignat}{2}
b_+(|x|)&\leq  a (x)\leq b^+(|x|), &&\quad\qquad x\in \X;\tag{B1}\label{condB1}\\
v_\circ(x)&\leq u_0(x)\leq v^\circ(x), &&\quad\qquad  x\in \X.\tag{B2}\label{condB4}
\end{alignat}
In order to control the allowed range of values for $a$ or $u_0$ in \eqref{condB1}--\eqref{condB4}, we consider also the following definition.

\begin{definition}
Let $b_1,b_2:\R\to\R_+$ and, for some $\rho\geq0$, $b_i(s)>0$ for all $s\in[\rho,\infty)$, $i=1,2$. The functions $b_1$ and $b_2$ are said to be  
 {\em (asymptotically) log-equivalent}, if
 \begin{equation}\label{eq:log-equiv}
   \log b_1(s) \sim \log b_2(s), \quad s\to\infty.
 \end{equation}
\end{definition}

We will assume also that there exist constants $M,\delta,r>0$ and a point $x_0\in\X$ such that
\begin{alignat}{2}
b^+(s)&\leq \frac{M}{(1+s)^{d+\delta}}, \quad\qquad && s \geq r;\tag{B2}\label{condB2}\\
b_+(s)&\geq \delta, \quad\qquad &&s\in[0,\delta];\tag{B3}\label{condB3}\\
v_\circ(x)&\geq \delta, \quad\qquad && x\in B_{\delta}(x_0).\tag{B5}\label{condB5}
\end{alignat}

Recall that our objective is to show \eqref{eq:whatweareproving1}--\eqref{eq:whatweareproving2} with $\La(t)=\La(t,c)$ given by \eqref{eq:defLa}. 
The choice of the `shape' for the function $c=c(x)$ will be determined by the initial condition $u_0:\X\to[0,\theta]$. We will distinguish two cases. Namely, if the initial condition is integrable,
\begin{equation}
u_0\in L^1(\X),\label{eq:intcase}
\end{equation}
then we will choose, for a suitable $b\in\Et$, 
\begin{equation}
  c(x)=b(|x|), \quad x\in\X. \tag{C1} \label{eq:radialc}
\end{equation}
On contrary, if the initial condition is separated from $0$ at $\bigtimes\limits_{j=1}^d\{-\infty\}$, i.e. if
\begin{equation}
\begin{gathered}
\textsl{there exist $\zeta\in(0,\theta]$ and $z\in\X$ such that}\\
u_0(x)\geq \zeta \1_{\nabla(z)}(x), \quad x\in\X,
\end{gathered}\label{eq:moncase}
\end{equation}
where
\begin{equation}
    \nabla(z):=\{y\in\X: y_j\leq z_j, \ 1\leq j\leq d\}, \quad z\in\X, \label{eq:nabla}
\end{equation}
then we will choose (again, for some $b\in\Et$)
\begin{equation}
c(x)=\int_{\Delta(x)}b(|y|)\,dy, \quad x\in\X. \tag{C2} \label{eq:integralc}
\end{equation}

We~will refer to the examples of $c(x)$ in \eqref{eq:radialc} and \eqref{eq:integralc} as to integrable and `monotone' case, respectively. Note that, in the `monotone' case, $c$ decays to $0$ along all coordinate axes in~$\X$. It is worth noting also that we have shown in \cite[Proposition~5.7]{FT2017a}, that if $u_0$ decays along a direction in $\X$, then the corresponding solution decays along this direction as well.

We are ready to formulate now the general result.
\begin{theorem}\label{thm:combined}
Let either \eqref{assum:kappa>m}--\eqref{firstfullmoment} hold or \eqref{assum:approx_of_basic} hold. Let \eqref{condB1}--\eqref{condB5} hold. Suppose that $u_0\in E_\theta^+$, $u_0\not\equiv0$, $\theta-u_0\not\equiv0$ (cf.~Remark~\ref{rem:agreementnonzero}); and let $u$ be the corresponding solution to \eqref{eq:basicequation}. 
\begin{enumerate}
  \item Suppose that both functions $b_+$ and $b^+$ in \eqref{condB1} belong to $\Dt$ and are log-equivalent to a function $b\in\Et $. Let, additionally, $b_+$ be long-tailed and tail-log-convex.
  \begin{enumerate}
    \item Suppose that  
    \begin{equation}\label{eq:asdsdps}
 u_0(x)\leq b^+(|x|),  \quad x\in\X.
 \end{equation}
Then \eqref{eq:whatweareproving1}--\eqref{eq:whatweareproving2} hold with $\La(t)=\La(t,c)$ given by \eqref{eq:defLa} where $c(x)$ is given by \eqref{eq:radialc}.
 \item Suppose that 
 \eqref{eq:moncase} holds and
  \begin{equation}\label{eq:hevisidelb}
  u_0(x)\leq \int_{\Delta(x)}b^+(|y|)\,dy, \quad x\in\X,
  \end{equation}  
where $\Delta(x)$ is given by \eqref{eq:bigDelta}. 
  Then \eqref{eq:whatweareproving1}--\eqref{eq:whatweareproving2} hold with $\La(t)=\La(t,c)$ given by \eqref{eq:defLa} where $c(x)$ is given by \eqref{eq:integralc}. 
  \end{enumerate}
  \item Consider functions $b_\circ,b^\circ\in\Dt$ which are both log-equivalent to a function $b\in\Et $. Let,~additionally, $b_\circ$ be long-tailed and tail-log-convex, and assume that, cf.~\eqref{condB1},
\begin{equation}
  a(x)\leq b^\circ(|x|), \quad x\in\X.
\label{eq:asdsdps2-mod}
\end{equation}
\begin{enumerate}
  \item Suppose that 
    \[
      v^\circ(x)=b^\circ(|x|), \quad v_\circ(x)=b_\circ(|x|), \quad x\in\X. 
    \]
    Then \eqref{eq:whatweareproving1}--\eqref{eq:whatweareproving2} hold with $\La(t)=\La(t,c)$ given by \eqref{eq:defLa} where $c(x)$ is given by \eqref{eq:radialc}.
   \item Suppose that 
    \[
      v^\circ(x)=\int_{\Delta(x)}b^\circ(|y|)\,dy, \quad v_\circ(x)=\int_{\Delta(x)}b_\circ(|y|)\,dy, \quad x\in\X.
    \]
      Then \eqref{eq:whatweareproving1}--\eqref{eq:whatweareproving2} hold with $\La(t)=\La(t,c)$ given by \eqref{eq:defLa} where $c(x)$ is given by \eqref{eq:integralc}. 
\end{enumerate}
\end{enumerate}
\end{theorem}

Theorem~\ref{thm:combined} will be proved in Subsection~\ref{subsec:corandex} below.
Note that the items 1(a) and 2(a) of Theorem~\ref{thm:combined} correspond to the case \eqref{eq:intcase}, and two others items correspond to \eqref{eq:moncase}.

\begin{remark}\label{rem:radsym}
In view of Lemma~\ref{le:GisOK}, Theorem~\ref{thm:combined} implies Theorem~\ref{thm:radsym}. The only observation which might to be mentioned here is that $q(s)\geq \delta$, $s\in[0,\rho]$, yields, cf.~\eqref{eq:nabla},
\[
\int_{\Delta(x)}q(|y|)dy\geq \mathrm{const}\cdot\1_{\nabla(0)}(x), \quad x\in\X.
\]
\end{remark}

\section{Scheme of the proof}\label{sec:stratofproof}
In this Section, we describe the scheme of the proof for Theorem~\ref{thm:combined}.
The detailed proof is presented in Subsection~\ref{subsec:corandex} below. 

We assume that \eqref{assum:kappa>m}--\eqref{assum:sufficient_for_comparison} hold. 
Let $u_0\in E^+_\theta$ and $u$ be the corresponding nonnegative solution to \eqref{eq:basicequation} according to Theorem~\ref{thm:existandcompared}. 

\subsection{Upper estimates}
Let $w$ be the solution to the linear problem \eqref{eq:majorequation},
\begin{equation}\label{eq:majorequation}
\begin{cases}
\dt w(x,t)=\ka(a*w)(x,t)-m w(x,t),\\[2mm]
w(x,0)=u_0(x).
\end{cases}
\end{equation}
The unique classical solution to \eqref{eq:majorequation} is given by $w(x,t)=e^{-m t}\bigl(e^{\ka tA}u_0\bigr)(x)$, where $Av:=a*v$, $v\in E$ is a bounded operator on~$E$. By~Theorem~\ref{thm:existandcompared}, $u(\cdot,t)\in E_\theta^+$ for $t\geq0$. Then, by~\eqref{assum:Gpositive}, $(Gu)(\cdot,t)\geq0$ for $t\geq0$, and we have, by e.g. \cite[Lemma~3.3.2]{Hen1981}, that
\begin{equation}\label{eq:majorisedbylinear}
u(x,t)\leq w(x,t), \quad x\in\X,\ t\geq0.
\end{equation} 

Below we explain the scheme of the proof for 
a reinforced version of \eqref{eq:whatweareproving2} with $u$ replaced by $w$. Namely, the following statement holds.

\begin{theorem}\label{thm:lin:est_above}
Let \eqref{assum:kappa>m} hold. Let $0\leq u_0\in E$ and $w=w(x,t)$ be the corresponding solution to the linear equation \eqref{eq:majorequation}. 
Let $b_1,b_2\in\Dt$ be both log-equivalent to a function $b\in\Et $.
Suppose that
\begin{equation}
a(x)\leq b_1(|x|), \quad x\in\X. \label{ass:newA5}
\end{equation}
Suppose also that either
\begin{equation}\label{eq:u0leqb2}
  u_0(x)\leq b_2(|x|), \ x\in\X,
\end{equation}
and $c:\X\to\R_+$ is given by \eqref{eq:radialc}, or
\begin{equation}\label{eq:u0leqintb2}
  u_0(x)\leq \int_{\Delta(x)} b_2(|y|)dy, \ x\in\X,
\end{equation}
and $c$ is given by \eqref{eq:integralc}. 
Then, for any small enough $\eps>0$, there exist $C_\eps,\tau_\eps>0$ such that
\begin{equation}\label{eq:theneedeupperforlin}
\esssup_{{x\in\X\setminus\La(t+\eps t,c)}} w(x,t)\leq C_\eps e^{-\frac{\eps\beta}{4}t}, \qquad t\geq \tau_\eps,
\end{equation} 
where $\La(t,c)$ is given by \eqref{eq:defLa}.
\end{theorem}
The proof of Theorem~\ref{thm:lin:est_above} is presented in Subsection~\ref{subsec:convtozero} and it is based on the arguments below. Firstly, we prove the following statement.
\begin{proposition}\label{prop:convenienttoexplain}
Let \eqref{ass:newA5} holds with a function $b_1\in\Dt$ which is log-equivalent to a function $b\in\Et $. Then there exists $\alpha_1\in(0,1)$ such that, for all $\alpha\in(\alpha_1,1)$, $b^\alpha\in\Et $, and both functions
\begin{equation}\label{eq:defofcalpha}
 \quad c_\alpha(x) := b(|x|)^\alpha,\ x\in\X; \qquad  c_\alpha(x):=\displaystyle\int_{\Delta(x)} b(|y|)^\alpha dy,\ x\in\X,
\end{equation}
satisfy the inequality
\begin{equation} \label{eq:aconvomegaoveromega12}
\limsup_{\la\to0+}\sup_{x\in\{c_\alpha<\la\}}\dfrac{(a*c_\alpha)(x)}{c_\alpha(x)}\leq1.
\end{equation}
\end{proposition}

Here and in the sequel, for $p:\X\to\R_+$ and $\la>0$, we denote
\begin{equation}\label{eq:levelset}
  \{p<\la\}:=\{x\in\X\mid p(x)<\la\}.
\end{equation}

Then, by usage of Proposition~\ref{prop:ba:suff_cond1}, we will receive that, for the function 
 \begin{equation}\label{eq:defofomegala1}
    c_{\alpha,\la}(x):=\min\bigl\{\la,c_\alpha(x)\bigr\}, \quad x\in\X, \ \la>0,
  \end{equation}
the following analogue of \eqref{eq:aconvomegaoveromega12} holds globally: for any $\delta>0$, there exists $\la>0$ such that
\begin{equation*}
\dfrac{(a*c_{\alpha,\la})(x)}{c_{\alpha,\la}(x)}\leq 1+\delta, \quad x\in\X.
\end{equation*}
From this, using Proposition~\ref{prop:ba:bound_above_gen}, we will conclude that the following analogue of \eqref{eq:theneedeupperforlin} holds, for any small enough $\eps>0$, 
\begin{equation}\label{eq:theneedeupperforlind}
\esssup_{{x\in\X\setminus\La(t+\eps t,c_\alpha)}} w(x,t)\leq C_\eps e^{-\frac{\eps\beta}{2}t}, \qquad t\geq \tau_\eps,
\end{equation}
with some $C_\eps,\tau_\eps>0$. Finally, by Proposition~\ref{prop:sassaf}, for small enough $\eps>0$, for large enough $t>0$, and for some $\alpha=\alpha(\eps)\in(\alpha_1,1)$, we will get that
\begin{equation}\label{eq:omegaincltoc}
  \La\Bigl(t+\frac{\eps t}{2} ,c_\alpha\Bigr) \subset \La(t+\eps t,c),
\end{equation}
where $c_\alpha$ is given by either of \eqref{eq:defofcalpha} and $c$ is given by \eqref{eq:radialc} or \eqref{eq:integralc}, respectively. Combining \eqref{eq:theneedeupperforlind} and \eqref{eq:omegaincltoc}, we obtain \eqref{eq:theneedeupperforlin}.

\subsection{Lower estimates}

Consider the reaction-diffusion form \eqref{eq:RDequation}--\eqref{eq:Markovgenerator}
of \eqref{eq:basicequation}. Because of the property~\eqref{eq:propertiesofF}, $(Fu)(\cdot,t)\geq0$ for $t\geq0$. Then, by the same arguments which implied \eqref{eq:majorisedbylinear}, we obtain that $u(x,t)\geq v(x,t)$ for $x\in\X$, $t\geq0$, where $v$ solves the following linear initial value problem, cf.~\eqref{eq:majorequation},
\begin{equation*}
\begin{cases}
\dt v(x,t)=\ka(a*v)(x,t)-\ka v(x,t),\\[2mm]
v(x,0)=u_0(x).
\end{cases}
\end{equation*}
Then, clearly, $v(x,t)=e^{-\ka t}\bigl(e^{\ka tA}u_0\bigr)(x)$ with $Av=a*v$ for $v\in E$, and since, we recall, $A$ is a bounded operator on $E$, we easily conclude that
\begin{equation*}
u(x,t)\geq v(x,t)\geq \ka t e^{-\ka t} (a*u_0)(x), \quad x\in\X, \ t\in\R_+.
\end{equation*}
More detailed arguments can be found in the proof of the following Theorem in Subsection~\ref{subsec:convtotheta}.

\begin{theorem}\label{thm:bb:est_below}
Let either \eqref{assum:kappa>m}--\eqref{firstfullmoment} hold or \eqref{assum:approx_of_basic} hold.  Let $u_0\in E_\theta^+$, $u_0\not\equiv0$, cf.~Remark~\ref{rem:agreementnonzero}; and let $u=u(x,t)$ be the corresponding solution to \eqref{eq:basicequation}. Let $b\in\Dt$ be long-tailed and tail-log-convex (in particular, let $b\in\Et $), and suppose that
either \eqref{eq:intcase} holds and $c(x)$ is given by \eqref{eq:radialc} or \eqref{eq:moncase} holds and $c(x)$ is given by \eqref{eq:integralc}.
Suppose also that 
\begin{equation}\label{eq:condonparambelow}
  (a*u_0)(x)\geq c(x), \quad x\in\X.
\end{equation}
Then, for each $\eps\in(0,1)$, the convergence \eqref{eq:whatweareproving1} holds.
  \end{theorem}

Before an explanation of the scheme for the proof of Theorem~\ref{thm:bb:est_below}, we discuss how to use it to prove the part of Theorem~\ref{thm:combined} regarding the convergence \eqref{eq:whatweareproving1}.

Let  \eqref{eq:intcase} hold, that corresponds to items 1(a) and 2(a) of Theorem~\ref{thm:combined}. It is assumed there that one of the functions $a$ and $u_0$ is estimated from below by $b_\sharp(|x|)$, $x\in\X$ (where $\sharp$ was $+$ and $\circ$, respectively) with a long-tailed and tail-log-convex $b_\sharp\in\Dt$, and another function is from $L^1(\X)$. Then, by Proposition~\ref{prop:liminfbelow}, we will get that 
\begin{equation}
 (a*u_0)(x)\geq D b_\sharp(|x|)=:\tilde{c}(x), \quad x\in\X,\label{eq:cx1}
 \end{equation}
 with some $D>0$.   

Let now \eqref{eq:moncase} hold. 
The item 1(b) of Theorem~\ref{thm:combined} can be easily reduced to the case $z=0\in\X$ in \eqref{eq:hevisidelb}.
Then, by \eqref{condB1}, we obtain that, cf.~\eqref{eq:nabla},
\begin{align}
( a *u_0)(x)&\geq \zeta \int_\X b_+(|y|)\1_{\nabla(0)}(x-y)\,dy\notag
\\&=\zeta \int_{\Delta(x)} b_+(|y|)\,dy=: \tilde{c}(x), \quad x\in\X.\label{eq:cx2}
\end{align}
Finally, in item 2(b) of Theorem~\ref{thm:combined}, by the first inequality in \eqref{condB4}, we have, denoting $p_\circ(x):=b_\circ(|x|)$, $x\in\X$, that
\begin{align}
( a *u_0)(x)&\geq \int_\X  a (x-y)\int_{\Delta(y)} p_\circ(z)\,dz\,dy\notag \\& =
\int_{\Delta(x)} ( a *p_\circ)(z)\,dz  \geq D \int_{\Delta(x)} b_\circ(|z|)\,dz =:\tilde{c}(x), \quad x\in\X,\label{eq:cx3}
\end{align}
with some $D>0$, where we used again Proposition~\ref{prop:liminfbelow} (recall that here $b_\circ\in\Dt$ is supposed to be long-tailed and tail-log-convex).

Then, for either of functions $\tilde{c}(x)$ given by \eqref{eq:cx1}--\eqref{eq:cx3}, one can apply Theorem~\ref{thm:bb:est_below} to get \eqref{eq:whatweareproving1} with $c$ replaced by $\tilde{c}$ and $\eps$ replaced by $\frac{\eps}{2}$. Finally, Proposition~\ref{prop:wehavethis} yields that, if $b_+$ or $b_\circ$ above is log-equivalent to a function $b\in\Dt$, then, for small enough $\eps>0$ and for large enough $t>0$, 
\begin{equation*}
\La(t-\eps t, c)\subset \La\Bigl(t-\frac{\eps}{2} \, t,\tilde{c}\Bigr), 
\end{equation*}
cf.~\eqref{eq:omegaincltoc}, where $c$ is given by either \eqref{eq:radialc} or \eqref{eq:integralc} for the cases \eqref{eq:intcase} and \eqref{eq:moncase}, respectively; and therefore, one gets \eqref{eq:whatweareproving1}.

And now we are going to outline the scheme of the proof for Theorem~\ref{thm:bb:est_below}. Take any $0<\delta<\beta=\ka-m$. Suppose that $\varpi:\X\times\R_+\to\R_+$ is a bounded function such that the function $v(x,t)=\la \varpi(x,t)$ is a sub-solution to the equation
\begin{equation}\label{eq:linmod}
\dt w(x,t)=\ka (a*w)(x,t)-(m+\delta)w(x,t)
\end{equation}
for all small enough $\la>0$, cf.~\eqref{eq:majorequation} and Definition~\ref{def:subsollin} below. Next, the assumption \eqref{assum:Glipschitz} implies the continuity of $G$ at $0$ on $\{0\leq v\in E\}$. Therefore, for any small enough $\la>0$, we obtain that $(Gv)(\cdot,t)\leq \delta$ for $t\geq0$. As a result,
\[
  \dt v - \ka a* v +m v+ v\,Gv\leq \dt v - \ka a* v+(m+\delta) v\leq 0
\] 
for large $t$, i.e. $v$ (with small $\la>0$) will be a sub-solution to \eqref{eq:basicequation} as well. See Proposition~\ref{prop:bb:subsol} for further details.

Next, in Proposition~\ref{prop:bb:subsollin}, we will show that the function 
\begin{align*}
 v(x,t)&=\frac{1}{\sigma} \int_{t}^{t+\sigma} g(x,s)\,ds, \\ \shortintertext{where} 
 g(x,t)&=\la \min\bigl\{1,c(x)e^{\beta (1-\eps)t}\bigr\},
 \end{align*}
cf.~\eqref{eq:defofomegala1}, will be a sub-solution to \eqref{eq:linmod} for small $\la,\eps,\sigma>0$, provided that $c(x)$ is given by either of \eqref{eq:radialc} or \eqref{eq:integralc} with a long-tailed and tail-log-convex function $b\in\Dt$. Then, by the above, such $v$ is a sub-solution to \eqref{eq:basicequation}. From this, by the comparison Theorem~\ref{thm:comparison}, we conclude that
\[
	u(x,t)\geq \la, \quad x\in\La(t-\eps t,c)
\]
for large $t>0$ and small $\la,\eps>0$. 
Finally, we cover $\La(t-\eps t,c)$ by compacts and apply the hair-trigger Theorem~\ref{thm:hairtrigger} on each of them.
For the (quite technical) details, see the proof of Theorem~\ref{thm:bb:est_below} in Subsection~\ref{subsec:convtotheta} below. 

\section{Technical tools}\label{sec:techtools}

Through this Section, $\beta>0$ is a fixed constant.

\subsection{Functions constructed by tail-decreasing functions}

\begin{definition}\label{def:beforeproofs}
If for some $b\in\Dt$ (c.f. Definition~\ref{def:super-subexp}), a function  $c=c(x)$
is given by either of \eqref{eq:radialc} and \eqref{eq:integralc}, we will say that the function $c$ is \emph{constructed} by the function $b$.
\end{definition}

For any $x=(x_1,\ldots,x_d)\in\X$, we set
\begin{align}
\m{x}&:=\max\limits_{1\leq j\leq d}x_j\in\R,\label{eq:mx}
\end{align}

\begin{remark}\label{re:coolssls}
If $c=c(x)$ is constructed by $b\in\Dt$, then evidently the set $\La(t,c)$ given by \eqref{eq:defLa} is nonempty for big enough~$t$, and the following limit holds, 
\[
	\lim_{t\to\infty} \sup_{x\in\X\setminus\La(t,c)}c(x) = \lim_{t\to\infty} 
  \sup_{x\in\{c<e^{-\beta t}\}}c(x) = 0,
\]
cf.~\eqref{eq:levelset}.
For $c$ given by \eqref{eq:radialc}, we have that $\lim\limits_{|x|\to\infty} c(x)=0$.
For $c$ given by \eqref{eq:integralc}, we have, cf.~\eqref{eq:mx},
\[
\lim_{\m{x}\to\infty}c(x)=\lim_{\m{x}\to\infty}\int_{\Delta(x)}b(|y|)dy=0.
\]
\end{remark}

The proof of the following Proposition is straightforward, cf. Definition~\ref{def:super-subexp} and \eqref{eq:defLa}.
\begin{proposition}\label{prop:dssdsdadassaqw}
	Let $c^{(i)}=c^{(i)}(x)$ be constructed by $b_i\in\Dt$, $i=1,2$ (both given simultaneously either by \eqref{eq:radialc} or by \eqref{eq:integralc}).
Suppose that there exists $\rho>0$ such that $b_1(s)\leq b_2(s)$ for all $s\geq\rho$.
Then, there exists $\tau=\tau(b_1,b_2)>0$ such that $\La(t,c^{(1)})\subset\La(t,c^{(2)})$ for all $t\geq \tau$.
In particular, if $b_1(s)=b_2(s)$ for all $s\geq\rho$, then $\La(t,c^{(1)})=\La(t,c^{(2)})$ for all $t\geq \tau$.
\end{proposition} 

The following proposition implies that if $c$ is constructed by $b$, then in terms of $\La(t,c)$ there is no loss of generality assuming that $b\in\Dt$ is strictly decreasing \textit{on the whole} $\R_+$.
\begin{proposition}\label{prop:new} 
 Let $c=c(x)$ be constructed by $b\in\Dt$, cf.~Definition~\ref{def:beforeproofs}. Then there exist functions $c_1$ and $c_2$, constructed by strictly decreasing on the whole $\R_+$ functions from $\Dt$, such that $c_1(x)\leq c(x)\leq c_2(x)$, $x\in\X$, and there exists $\rho\geq0$ such that 
\begin{align*}
	&c_1(x) = c(x) = c_2(x), \quad |x|\geq \rho, \quad \text{if $c_1,c,c_2$ are given as in \eqref{eq:radialc}},\\ 
	&c_1(x) = c(x) = c_2(x), \quad \m{x}\geq \rho, \quad \text{if $c_1,c,c_2$ are given as in \eqref{eq:integralc}}.
\end{align*}
As a result, there exists $\tau=\tau(c,c_1,c_2)\geq0$ such that
\[
	\La(t,c)=\La(t,c_1)=\La(t,c_2), \quad t\geq \tau.
\]
\end{proposition}
\begin{proof}
Let $c$ be constructed by a $b\in\Dt$. 
By Definitions~\ref{def:right-sideclasses}--\ref{def:super-subexp} there exists $\rho'>0$ such that $b$ is decreasing on $(\rho',\infty)$ to $0$ and, for some $D_1,D_2>0$, $D_1\leq b(s)\leq D_2$ for $s\in[0,\rho']$.
Choose $\rho\geq\rho'$ such that $b(\rho)\leq D_1$.
Set $b_1(s)=b_2(s)=b(s)$ for $s>\rho$ and define $b_1$ on $[0,\rho]$ as an arbitrary decreasing function with $b_1(0)\leq D_1$.
Similarly, we define $b_2$ on $[0,\rho]$ as an arbitrary decreasing bounded function with $b_2(\rho)\geq D_2$.
As a result, $b_1(s)\leq b(s)\leq b_2(s)$, $s\in\R_+$.
Let $c_1,c_2$ be constructed by $b_1,b_2\in\Dt$, which are strictly decreasing on $\R_+$ such that either $c_1,c_2,c$ are all given by \eqref{eq:radialc} or $c_1,c_2,c$ are all given by \eqref{eq:integralc}. Then in both cases, evidently, $c_1(x)\leq c(x)\leq c_2(x)$, $x\in\X$. The rest of the proof follows from Proposition~\ref{prop:dssdsdadassaqw}.
\end{proof}

Let $b,b^{\alpha}\in\Dt$ for some $\alpha\in(0,1)$. We denote by $c_\alpha$ the function constructed by $b^\alpha$, as in \eqref{eq:defofcalpha}.

\begin{remark}\label{rem:saddssdadsasadsda1}
It is easy to see that, if for some $\alpha_0\in(0,1)$, $b^{\alpha_0}\in\Dt$ is strictly decreasing on $\R_+$, then $b^\alpha\in\Dt$ for all $\alpha\in[\alpha_0,1]$.
\end{remark}

\begin{proposition}\label{prop:sassaf}
For any $\alpha_0\in\bigl(\frac{3}{4},1\bigr)$ there exists $\eps_0=\eps_0(\alpha_0)\in(0,1)$ such that, for any $\eps\in(0,\eps_0)$, there exists $\alpha=\alpha(\eps)\in(\alpha_0,1)$ such that the following holds.
For any $b\in\Dt$ strictly decreasing on $\R_+$ such that $b^{\alpha_0}\in\Dt$, let $c$ and $c_\alpha$ be constructed by $b$ and $b^\alpha$, respectively. Then there exists 
 $\tau=\tau(\eps,b)>0$ such that, for any $t\geq\tau$,
\begin{align}
\label{eq:supset-}
\La(t{-}\eps t,c_\alpha)\subset \La\Bigl(t{-}\frac{\eps t}{2},c\Bigr),\\
\label{eq:supset+}
\La\Bigl(t{+}\frac{\eps t}{2},c_\alpha\Bigr)\subset \La(t{+}\eps t,c).
\end{align}
\end{proposition}
\begin{proof}
	We will prove \eqref{eq:supset+}. The proof of \eqref{eq:supset-} is fully analogous. Consider two cases \eqref{eq:radialc} and \eqref{eq:integralc} separately.

	1) Let $c$ be given by \eqref{eq:radialc}. Since $\alpha_0\in\bigl(\frac34,1\bigr)$, one can define $\eps_0:=\frac{1-\alpha_0}{\alpha_0-\frac{1}{2}}\in(0,1)$.
Take an arbitrary $\eps\in(0,\eps_0)$, then one easily has that
\begin{equation}\label{eq:alpha0alpha+}
  \alpha:=\dfrac{1+\frac{\eps}{2}}{1+\eps}\in(\alpha_0,1).
\end{equation}
By \eqref{eq:explicitLa}, for some $\tau>0$, the equality in \eqref{eq:supset+} is just equivalent to
\[
	\eta(t{+}\frac{\eps t}{2}, b^\alpha)=\eta(t{+}\eps t, b), \quad t\geq \tau.
\]
To prove the latter equality, apply $\log b^\alpha=\alpha\log b$ to both its parts:
  \[
    -\Bigl(1+\frac{\eps}{2}\Bigr)\beta t=-\alpha (1+\eps)\beta t,
  \]
  that is equivalent to \eqref{eq:alpha0alpha+}.

	2) Let $c$ be given by \eqref{eq:integralc}. Prove the following inequality, which is equivalent to \eqref{eq:supset+},
\begin{equation}\label{eq:supset}
	\X\setminus\La(t{+}\eps t,c)\subset\X\setminus \La\Bigl(t{+}\frac{\eps t}{2},c_\alpha\Bigr),\quad t\geq \tau.
\end{equation}
Recall that the inclusion $x\in\X\setminus\La(t{+}\eps t,c)$ is equivalent to
  \begin{equation}\label{eq:whatdowehave}
c(x)=\int_{\Delta(x)} b(|y|)dy < e^{-\beta(1+\eps)t}.
  \end{equation}
  We will use H\"older's inequality to estimate $c_\alpha(x)$.
  It is easy to see that the function
  \[
  f(\alpha):=\alpha-\sqrt{\alpha(1-\alpha)}:\bigl(\tfrac{1}{2},1\bigr)\to(0,1)
  \]
  is increasing. We set $p:=p(\alpha):=\frac{1}{f(\alpha)}>1$ and $q:=q(\alpha):=\frac{1}{1-f(\alpha)}>1$. Then $\frac{1}{p}+\frac{1}{q}=1$ and, by \eqref{eq:whatdowehave}, we have
  \begin{align}
  c_\alpha(x)&=\int_{\Delta(x)} b(|y|)^{f(\alpha)+(\alpha-f(\alpha))} dy\notag\\&\leq
  \biggl(\int_{\Delta(x)} b(|y|)^{f(\alpha)p}dy\biggl)^{\frac{1}{p}}
  \biggl(\int_{\Delta(x)} b(|y|)^{(\alpha-f(\alpha))q}dy\biggl)^{\frac{1}{q}}\notag\\
  & < e^{-\beta(1+\eps)f(\alpha)t}
  \biggl(\int_{\Delta(x)} b(|y|)^{(\alpha-f(\alpha))q}dy\biggl)^{\frac{1}{q}}\label{eq:tempweneed}
  \end{align}
	To get the finiteness of the latter integral in \eqref{eq:tempweneed}, it is enough to have there $\alpha$ such that $\alpha_0<g(\alpha)<1$, where
   \[
   g(\alpha):=(\alpha-f(\alpha))q(\alpha)=\frac{\sqrt{\alpha}}{\sqrt{\alpha}
   +\sqrt{1-\alpha}}, \quad \alpha\in\bigl(\tfrac{1}{2},1\bigr).
   \]
  It is easy to see that $g:\bigl(\frac{1}{2},1\bigr)\to\bigl(\frac{1}{2},1\bigr)$ is increasing and $g(\alpha)<\alpha$, $\alpha\in \bigl(\frac{1}{2},1\bigr)$. Note also that $g\bigl(\frac{9}{10}\bigr)=\frac{3}{4}$. As a result, for the given $\alpha_0\in\bigl(\frac{3}{4},1\bigr)$, there exists a unique $\alpha_1\in\bigl(\frac{9}{10},1\bigr)$ such that $\alpha_0=g(\alpha_1)<\alpha_1$. Hence, for any $\alpha\in(\alpha_1,1)\subset(\alpha_0,1)$, one gets $g(\alpha)>g(\alpha_1)=\alpha_0$, and then $\int_\X b(|y|)^{g(\alpha)} dy<\infty$; in particular, the latter integral in \eqref{eq:tempweneed} is finite.

  Next, the function $h(\eps)=\dfrac{1+\frac{\eps}{2}}{1+\eps}:(0,1)\to\bigl(\frac{3}{4},1\bigr)$ is decreasing; cf.~\eqref{eq:alpha0alpha+}. Therefore, there exists a unique $\eps_0\in(0,1)$ such that $h(\eps_0)=\alpha_1$; then we have $h:(0,\eps_0)\to(\alpha_1,1)$. Take and fix now an arbitrary $\eps\in(0,\eps_0)$. Since,
  \[
  f:(\alpha_1,1)\to\bigl(f(\alpha_1),1\bigr)\subset(\alpha_1,1)=(h(\eps_0),1)
  \]
  is increasing (we used here that $f(\alpha)<\alpha$), there exists a unique $\alpha=\alpha(\eps)\in (\alpha_1,1)$ such that
  \begin{equation}\label{eq:defofalpha}
    f(\alpha)=h(\eps)=\dfrac{1+\frac{\eps}{2}}{1+\eps}.
  \end{equation}

Therefore, after $\eps_0,\eps,\alpha$ are chosen, we take an arbitrary $b\in\Dt$ such that $b^{\alpha_0}\in\Dt$, and let $c$ be constructed by $b$. 
  For this $\alpha$, by the above, $\int_\X b(|y|)^{g(\alpha)} dy<\infty$; therefore, there exists $r>0$ such that, for all $x\in\X$ with $\m{x}>r$,
  \[
  \int_{\Delta(x)} b(|y|)^{g(\alpha)}dy\leq 1.
  \]
  The latter inequality together with \eqref{eq:defofalpha} and \eqref{eq:tempweneed} implies that
\begin{equation}\label{eq:whathaveweobtained}
  c_\alpha(x)\leq e^{-\beta(1+\frac{\eps}{2})t},
\end{equation}
provided that $x\in \X\setminus\La(t{+}\eps t,c)$ (i.e. \eqref{eq:whatdowehave} holds) and $\m{x}>r$. In~\eqref{eq:whatdowehave}, $\m{x}\to\infty$ if and only if $t\to\infty$; cf.~Remark~\ref{re:coolssls}. Therefore, there exists $\tau=\tau(r)=\tau(\eps,b)>0$ such that $t\geq \tau$ in \eqref{eq:whatdowehave} implies $\m{x}\geq r$. As a result, for any $t\geq \tau$ and any $x\in \X\setminus\La(t{+}\eps t,c)$, one gets \eqref{eq:whathaveweobtained}, that means that $x\in \X\setminus\La\bigl(t{+}\frac{\eps t}{2},c_\alpha\bigr)$; i.e. \eqref{eq:supset} holds.
\end{proof}

\begin{proposition}\label{prop:wehavethis}
	Let $b_1,b_2\in\Dt$ be log-equivalent functions such that, for some $\alpha_0\in\bigl(\frac{3}{4},1\bigr)$, $b_i^{\alpha_0}\in\Dt$, $i=1,2$. Let $c^{(i)}$ be constructed by $b_i$, $i=1,2$ (both satisfy simultaneously either \eqref{eq:radialc} or \eqref{eq:integralc}). Then there exists $\eps_0=\eps_0(\alpha_0)\in(0,1)$ such that, for any $\eps\in(0,\eps_0)$, there exists  $\tau=\tau(\eps)>0$ such that, for any $t\geq\tau$,
\begin{align}
\label{eq:supsetlog-}
\La(t{-}\eps t,c^{(1)})\subset \La\Bigl(t{-}\frac{\eps t}{2},c^{(2)}\Bigr),\\
\label{eq:supsetlog+}
\La\Bigl(t{+}\frac{\eps t}{2},c^{(1)}\Bigr)\subset \La(t{+}\eps t,c^{(2)}).
\end{align}
\end{proposition}
\begin{proof}
We assume first that both $b_1$ and $b_2$ are strictly decreasing on $\R_+$.
Let $\eps_0$ be given by Proposition~\ref{prop:sassaf}. Take an arbitrary $\eps\in(0,\eps_0)$ and consider $\alpha=\alpha(\eps)\in(\alpha_0,1)$ also given by Proposition~\ref{prop:sassaf}. 
Let $\rho_0>0$ be such that $b_i(\rho_0)\leq 1$, $i=1,2$. Set $\delta:=1-\alpha\in(0,1-\alpha_0)$. By~\eqref{eq:log-equiv}, there exists $\rho_\alpha\geq\rho_0$ such that
\[
1-\delta<\frac{-\log b_1(s)}{-\log b_2(s)}<1+\delta, \quad s>\rho_\alpha,
\]
in particular,
\begin{equation}\label{eq:ewrerwewrrewewrewr}
b_1(s)< b_2(s)^{\alpha}, \quad s>\rho_\alpha.
\end{equation} 
By~Remark~\ref{rem:saddssdadsasadsda1}, $b_2^\alpha\in\Dt$, and hence, by \eqref{eq:ewrerwewrrewewrewr} and Proposition~\ref{prop:dssdsdadassaqw}, applying to $b_1$ and $b_2^\alpha$, one gets 
\[
	\La\Bigl(t{+}\frac{\eps t}{2},c^{(1)}\Bigr)\subset \La\Bigl(t{+}\frac{\eps t}{2},c^{(2)}_\alpha\Bigr).
\]
The latter inequality together with \eqref{eq:supset+} for $c=c^{(2)}$ imply \eqref{eq:supsetlog+}.

Next, by \eqref{eq:ewrerwewrrewewrewr},
 $b_3(s):=b_1(s)^\frac{1}{\alpha}< b_2(s)$, if only $s>\rho_\alpha$. From here we have that $b_3\in\Dt$ and, moreover, 
by Proposition~\ref{prop:dssdsdadassaqw}, applying to $b_3$ and $b_2$,
\[
	\La\Bigl(t{-}\frac{\eps t}{2},c^{(3)}\Bigr)\subset \La\Bigl(t{-}\frac{\eps t}{2},c^{(2)}\Bigr),
\]
where $c^{(3)}$ in constructed by $b_3$.  The latter inequality together with \eqref{eq:supset-} for $c=c^{(3)}$ imply \eqref{eq:supsetlog-}.

Let now $b_i\in\Dt$, $i=1,2$ be arbitrary. Then, by the proof of~Proposition~\ref{prop:new}, there exist $\tilde{c}^{(i)}$ constructed by $\tilde{b}_i\in\Dt$, strictly decreasing on $\R_+$ such that $b_i(s)=\tilde{b}_i(s)$ for big enough $s$. Then $\tilde{b}_1$ and $\tilde{b}_2$ are log-equivalent. Applying the previous considerations to $\tilde{b}_i$, $i=1,2$, we get \eqref{eq:supsetlog-} and \eqref{eq:supsetlog+}, with $c^{(i)}$ replaced by $\tilde{c}^{(i)}$, $i=1,2$, for big enough $t$. Then, by Proposition~\ref{prop:new}, one gets the statement.
\end{proof}

\subsection{Functions constructed by long-tailed functions}
Recall that long-tailed functions were defined in Definition~\ref{def:right-sideclasses}. 

\begin{proposition}[\!\!{\cite[Lemma 4.1]{FT2017b}}]\label{prop:rightsidelongtailedmeansspatially}
	Let $c$ given by \eqref{eq:radialc} be constructed by a long-tailed function $b\in\Dt$.  Then, for any $r>0$,
  \begin{equation}\label{eq:long-tailed}
        \lim_{|x|\to\infty} \sup_{|y|\leq r} \biggl\lvert \frac{c(x+y)}{c(x)} - 1 \biggr\rvert = 0.
  \end{equation}
\end{proposition}

\begin{proposition}
	Let $c$ given by \eqref{eq:integralc} be constructed by a long-tailed function $b\in\Dt$. Then 
\begin{equation}\label{eq:monotonelong-tailed}
    \lim_{\m{x}\to\infty}\frac{c(x+h)}{c(x)}=1, \quad h\in\R_+^d.
  \end{equation}
\end{proposition}
\begin{proof}
Let $b$ be decreasing on $(\rho,\infty)$ for some $\rho>0$.
Fix an arbitrary $h\in\R_+^d$, $h\neq0$, and take any $R> \m{h}$.
Note that, for any $x,y\in\X$ such that $y_j\in[x_j,x_j+R]$, $1\leq j\leq d$, one has
  $\bigl\lvert |y| - |x|\bigr\rvert\leq |y-x|\leq R\sqrt{d}$.  
Assume now that $\m{x}\geq \rho+R\sqrt{d}$. Then, for any $y$ as above, 
  \[
  \biggl\lvert\frac{b(|y|)}{b(|x|)}-1\biggr\rvert \leq
  \sup_{|\tau|\leq R\sqrt{d}}\ 
    \biggl\lvert\frac{b(|x|+\tau)}{b(|x|)}-1\biggr\rvert \to 0, \quad |x|\to\infty,
  \]
  because of e.g. \cite[formula (2.18)]{FKZ2013}; cf. also \cite[Remark~2.1]{FT2017b}.

  Therefore, for any $\eps\in(0,1)$, there exists $r=r(\eps, R)>\rho+R\sqrt{d}$ such that, for all $x\in\X$ with $\m{x}\geq r$ (that, again, implies $|x|\geq r$), one has
  \[
  1-\eps \leq \frac{b(|y|)}{b(|x|)}\leq 1+\eps, \quad y_j\in[x_j,x_j+R], \ 1\leq j\leq d.
  \]
  As a result,
  \begin{align*}
  1-\frac{c(x+h)}{c(x)}
  &=\frac{\displaystyle\int_{x_1}^{x_1+h_1}\ldots\int_{x_d}^{x_d+h_d} b(|y|)\,dy}{\displaystyle\int_{x_1}^{\infty}\ldots\int_{x_d}^{\infty}
  b(|y|)\,dy}
  \leq \frac{\displaystyle\int_{x_1}^{x_1+\m{h}}\ldots\int_{x_d}^{x_d+\m{h}} \dfrac{b(|y|)}{b(|x|)}\,dy}{\displaystyle\int_{x_1}^{x_1+R}\ldots\int_{x_d}^{x_d+R} \dfrac{b(|y|)}{b(|x|)}\,dy}
  \\&\leq\frac{1+\eps}{1-\eps}\frac{\m{h}^d}{R^d}<\eps,
  \end{align*}
  provided that $R=R(\m{h},\eps)>\m{h}$ is chosen big enough. The statement is proved.
\end{proof}  
\begin{remark}
	Note that the previous result remains true if $c$ is defined by \eqref{eq:integralc} with $\Delta(x)$ replaced by $\Delta(x+x_0)$ for a fixed $x_0\in\X$.
\end{remark}

The following proposition gives a sufficient condition for \eqref{eq:condonparambelow}; the result is a generalization of \cite[Theorem~4.2]{FKZ2013}.
\begin{proposition}\label{prop:liminfbelow}
  Let $f\in L^1(\X\to\R_+)$ and $c:\X\to(0,\infty)$ be a bounded function such that \eqref{eq:long-tailed} holds. Then
  \begin{equation}\label{eq:straight}
    \liminf_{|x|\to\infty}\frac{(c*f)(x)}{c(x)}\geq \int_\X f(y)\,dy.
  \end{equation}
  Moreover, there exists $D>0$ such that
  \[
    (c*f)(x)\geq D c(x), \quad x\in\X.
  \]
\end{proposition}
\begin{proof}
  For any $r>0$, we have
  \begin{align*}
    \frac{(c*f)(x)}{c(x)}
    &\geq \biggl(1-\sup_{|y|\leq r}\Bigl\lvert
    \frac{c(x-y)}{c(x)}-1\Bigr\rvert \biggr)\int_{|y|\leq r}f(y)\,dy.
  \end{align*}
  Take an arbitrary $\delta\in(0,1)$ and choose $r=r(\delta)>0$ such that
  $\int_{|y|\leq r}f(y)\,dy>(1-\delta)\int_\X f(y)\,dy$.
  Next, by \eqref{eq:long-tailed}, there exists $\rho=\rho(r)=\rho(\delta)\geq r$ such that
  $
  \sup\limits_{|y|\leq r}\bigl\lvert
    \frac{c(x-y)}{c(x)}-1\bigr\rvert<\delta,
  $ for all $|x|\geq\rho$. As~a~result, for any $\delta\in(0,1)$, there exists $\rho=\rho(\delta)>0$ such that
  \[
  \frac{(c*f)(x)}{c(x)}> (1-\delta)^2\int_{\X}f(y)\,dy, \quad |x|\geq\rho,
  \]
  that yields \eqref{eq:straight}. Finally, by e.g. \cite[Lemma~2.1]{FKT100-1}, $c*f$ is a continuous function on $B_\rho(0)$; then, it is easy to see that $c(x)>0$, $x\in\X$ implies that $(c*f)(x)>0$, $x\in\X$. Hence the boundedness of $c$ yields
  $\inf\limits_{|x|\leq \rho} \frac{(c*f)(x)}{c(x)}>0$, that fulfilled the statement.
\end{proof}

\section{Proofs}\label{sec:proof}

In this Section, $\beta>0$ is given by \eqref{assum:kappa>m}.

\subsection{Proofs of upper estimates}\label{subsec:convtozero}

In this Subsection, we are going to prove Theorem~\ref{thm:lin:est_above}.

\subsubsection{Preliminaries}
For a  function $\tomega:\X\to(0,+\infty)$, we define, for any $f:\X\to\R$,
\begin{equation}\label{eq:defofomeganorm}
  \lVert f\rVert_\tomega:=\sup_{x\in\X} \frac{|f(x)|}{\tomega(x)}\in[0,\infty].
\end{equation}
If $\tomega(x)=b(|x|)$, $x\in\X$, for a function $b:\R_+\to(0,\infty)$, we will use the notation $\|f\|_b:=\|f\|_\tomega$.

\begin{proposition}[\protect{cf.~\cite[Propostion 3.1]{FKT100-3}}] \label{prop:ba:bound_above_gen}
 Let a function $\tomega:\X\to(0,+\infty)$ be such that $a*\tomega$ is well-defined (for example, let $\tomega$ be bounded) and, for some $\ga\in(0,\infty)$,
\begin{equation} \label{eq:ba:Phi_est_suff_cond}
\dfrac{(a*\tomega)(x)}{\tomega(x)}\leq\ga, \quad x\in\X.
\end{equation}
 Let $0\leq u_0\in L^\infty(\X)$ and $\normom{u_0}<\infty$; let $w=w(x,t)$ be the corresponding solution to \eqref{eq:majorequation}. Then
\begin{equation} \label{eq:ba:u_est_above_goal}
  \normom{w(\cdot,t)}\leq \normom{u_0}e^{(\ka(\ga-1) +\beta) t},\quad  t\geq0.
\end{equation}
\end{proposition}
\begin{proof}
The solution to \eqref{eq:majorequation} is given by
\begin{equation}\label{eq:explsoltolin}
  w(x,t)=e^{-mt}u_0(x)+e^{-mt}\sum_{n=1}^\infty \frac{(\ka t)^n}{n!}\bigl(a^{*n}*u_0\bigr)(x),
\end{equation}
for $x\in\X,\ t\geq0$, where $a^{*n}:=a*\ldots * a$ (the convolution is taken $n-1$ times). Since $\normom{u_0}<\infty$, we have 
$0\leq u_0(x)\leq \normom{u_0}\tomega(x)$, $x\in\X$. Next, \eqref{eq:ba:Phi_est_suff_cond} evidently implies 
\[
  \bigl(a^{*n}*\tomega\bigr)(x)\leq \ga^n\tomega(x), \quad x\in\X.  
\]
As a result, we get from \eqref{eq:explsoltolin} that
\begin{align*}
0\leq w(x,t)&\leq e^{-mt}\normom{u_0}\tomega(x)+e^{-mt}\normom{u_0}\sum_{n=1}^\infty \frac{(\ka t)^n}{n!}\ga^n\tomega(x)\\
&=\normom{u_0}\tomega(x) \Bigl(e^{-mt}+e^{-mt}\bigl(e^{\ka \ga t}-1\bigr)\Bigr)=\normom{u_0}\tomega(x)e^{(\ka \ga-m) t},
\end{align*}
that implies \eqref{eq:ba:u_est_above_goal}.

\end{proof}

\begin{remark}
In \cite[Propostion 3.1]{FKT100-3}, we considered, for an arbitrary $\la>0$ and a unit vector, $\xi\in\X$, the function $\tomega(x)=e^{-\la x\cdot\xi}$ (recall that $x\cdot\xi$ stands for the scalar product in $\X$). Then, clearly, $\frac{(a^+*\tomega)(x)}{\tomega(x)}\equiv \int_\X a^+(y) e^{\la y\cdot \xi}\,dy=:\ga$, provided that the latter integral is finite (that was the crucial assumption to get the constant speed of the front in \cite{FKT100-3}). Note that then \cite[Proposition~2.4]{FKT100-2} and \eqref{eq:ba:u_est_above_goal} implies that
$w(x,t)\leq \alpha_\xi e^{\beta_\xi t-\la_\xi x\cdot\xi}$, $x\in\X$, $t\geq0$ for some $\alpha_\xi,\la_\xi>0$, $\beta_\xi\in\R$.
\end{remark}

\begin{proposition}[\!\!\!{\cite[Proposition~2.4]{FT2017b}}]\label{prop:ba:suff_cond1}
  Let a function $\upomega:\X\to(0,+\infty)$ be~such that, for any $\la>0$, 
  \begin{equation}\label{eq:defofsetOmegala}
    \{\upomega<\la\}\neq\emptyset,
  \end{equation}
cf.~\eqref{eq:levelset}. Suppose further that 
\begin{equation} \label{eq:ba:23r23Phi_est_suff_cond}
	\eta:=\limsup_{\la\to0+}\sup_{x\in\{\upomega<\la\}}\dfrac{( a *\upomega)(x)}{\upomega(x)}\in(0,\infty).
\end{equation}
Then, for any $\delta\in(0,1)$, there exists $\la=\la(\delta,\upomega)\in(0,1)$ such that \eqref{eq:ba:Phi_est_suff_cond} holds, with
 \begin{equation}\label{eq:defofomegala}
    \tomega(x):=\upomega_\la(x):=\min\bigl\{\la,\upomega(x)\bigr\}, \quad x\in\X,
  \end{equation}
and $\ga:=\max\{1,(1+\delta)\eta\}$.
\end{proposition}

\begin{proposition}\label{prop:mainabove}
Let  $\upomega$ be constructed by $b\in\Dt$ (see Definition~\ref{def:beforeproofs}) and satisfy,
    cf.~\eqref{eq:ba:23r23Phi_est_suff_cond},
   \begin{equation} \label{eq:aconvomegaoveromega}
		 \limsup_{\la\to0+}\sup_{x\in\{\upomega<\la\}}\dfrac{(a*\upomega)(x)}{\upomega(x)}\leq1.
\end{equation}
   Let $0\leq u_0\in L^\infty(\X)$ be such that
$ \lVert u_0\rVert_{\upomega}<\infty$, cf.~\eqref{eq:defofomeganorm}, and let $w=w(x,t)$ be the corresponding solution to \eqref{eq:majorequation}. Then, for any $\eps\in(0,1)$, there exist $A_\eps>0$ and $t_0=t_0(\eps)>0$ such that
\begin{equation}\label{eq:fullineqabove}
	\esssup_{x\notin\La(t{+}\eps t,\upomega)} w(x,t)\leq \bigl(A_\eps+ \lVert u_0\rVert_{\upomega} \bigr) e^{-\frac{\eps\beta}{2}t}, \quad t\geq t_0.
\end{equation}
\end{proposition}

\begin{proof}
Take an arbitrary $\eps\in(0,1)$ and let $\delta=\delta(\eps)\in(0,1)$ be chosen later.
By Proposition~\ref{prop:ba:suff_cond1}, there exists $\la=\la(\delta,\upomega)=\la(\eps,\upomega)\in(0,1)$ such that  \eqref{eq:ba:Phi_est_suff_cond} holds, with $ \tomega$ given by \eqref{eq:defofomegala} and $\ga=1+\delta$. Set $\|u_0\|_\infty:=\|u_0\|_{L^\infty(\X)}$.
   Note that
   \begin{equation}\label{eq:uoiomegalambda}
		 \frac{u_0(x)}{\upomega_\la(x)}\leq \frac{\theta}{\la}\1_{\X\setminus\{\upomega<\la\}}(x)+
		 \frac{u_0(x)}{\upomega(x)}\1_{\{\upomega<\la\}}(x)
\leq \frac{\|u_0\|_{\infty}}{\la}+\lVert u_0\rVert_{\upomega}
<\infty,
\end{equation}
and one can apply Proposition~\ref{prop:ba:bound_above_gen}. Namely, setting $A_\eps:=\frac{\|u_0\|_{\infty}}{\la}>0$, one gets from
\eqref{eq:uoiomegalambda}, \eqref{eq:ba:u_est_above_goal} that, for a.a.~$x\in\{\upomega<\la\}$ and for all $t\geq0$,
  \begin{equation}
 w(x,t)\leq \lVert u_0\rVert_{\upomega_\la} e^{(\ka \delta+\beta)t}\upomega_\la(x)
    \leq \bigl(A_\eps+\lVert u_0\rVert_{\upomega}\bigr) e^{(\ka \delta+\beta)t}\upomega(x).\label{eq:crucial21213}
    \end{equation}
 By \eqref{eq:defofsetOmegala},
\begin{equation}\label{eq:simpleuseful}
	\X\setminus\La(t{+}\eps t,\upomega)=\{\upomega<e^{-\beta(t{+}\eps t)}\},\quad t>0.
\end{equation}
 Set
 $ t_0=t_0(\eps):=-\frac{1}{(1+\eps)\beta}\log\la>0$.
 One gets from \eqref{eq:simpleuseful} that, for any $t\geq t_0$, 
 \[
	 \X\setminus\La(t{+}\eps t,\upomega)\subset \X\setminus\La(t_0{+}\eps t_0,\upomega)=\{\upomega<\la\}.
 \]
 Hence, by \eqref{eq:crucial21213}, \eqref{eq:simpleuseful}, for a.a.~$x\in\X\setminus\La(t{+}\eps t,\upomega)$, one gets
    \begin{equation*}
  w(x,t)
    \leq \bigl(A_\eps+\lVert u_0\rVert_{\upomega}\bigr) e^{(\ka \delta+\beta)t}\upomega(x)
		\leq \bigl(A_\eps+\lVert u_0\rVert_{\upomega}\bigr) e^{(\ka \delta+\beta)t}e^{-\beta(t{+}\eps t)},
    \end{equation*}
    and
    \[
      \ka \delta+\beta  -\beta(1+\eps)=\ka \delta -\beta\eps= -\frac{\eps\beta}{2},
    \]
    if only we set from the very beginning $\delta:=\frac{\eps\beta}{2\ka  }$. The statement is proved.
\end{proof}
\begin{remark}
  It is easy to see from the proof above, that the denominator $2$ in the right-hand side of \eqref{eq:fullineqabove} can be changed on $1+\nu$, for an arbitrary $\nu\in(0,1)$; then $t_0=t_0(\eps,\nu)$.
\end{remark}

\subsubsection{\protect Proof of Proposition~\ref{prop:convenienttoexplain}}

We are going to show now that for the functions $\upomega=c_\alpha$ from \eqref{eq:defofcalpha} the inequality \eqref{eq:aconvomegaoveromega} holds.

\begin{proposition}\label{prop:newsimple}
Let $b\in\Et $. Then there exists $\alpha_1\in(0,1)$ such that, for all $\alpha\in[\alpha_1,1]$, $b^\alpha\in\Et $.
\end{proposition}
\begin{proof}
Let $d>1$ and $b\in\Et$. If $b$ is given by \eqref{eq:polynomialb}, then, for any $\alpha'\in\bigl(\frac{d}{d+\mu},1)$,
\begin{equation}
\int_0^\infty b(s)^{\alpha'} s^{d-1}\,ds<\infty. \label{eq:finite_d-1_momentalpha}
\end{equation}
If $b$ is such that, for all $\nu\geq1$, \eqref{eq:quicklydecreasing} holds, then, evidently, \eqref{eq:finite_d-1_momentalpha} holds for all $\alpha'\in(0,1)$.
For $d=1$ and $b\in\Et[1]$, \eqref{eq:finite_d-1_momentalpha} holds if only $\alpha'\in\bigl(\frac{1}{1+\delta},1)$, where $\delta$ is sufficiently small.
Then, by \cite[Theorem 3.1]{FT2017b}, there exists $\alpha_1\in(\alpha',1)$ such that, for all $\alpha\in[\alpha_0,1]$, $b^\alpha\in\Et$.
The proof is fulfilled.
\end{proof}

The following proposition ensures that \eqref{eq:aconvomegaoveromega} holds for $\upomega=c_\alpha$ given as in to \eqref{eq:radialc}.
\begin{proposition}[\!\!{\cite[Propositions 4.2, 4.3]{FT2017b}}]\label{prop:twoinone}
	Let \eqref{ass:newA5} hold with $b_1\in\Dt$ which is strictly decreasing on $\R_+$ and log-equivalent to a function $b\in\Et $. 
	Then there exists $\alpha_1\in(0,1)$ such that, for all $\alpha\in(\alpha_1,1)$,
 the function $\upomega(x)=b(|x|)^\alpha$, $x\in\X$, satisfies \eqref{eq:aconvomegaoveromega}.
\end{proposition}
Now we will show that \eqref{eq:aconvomegaoveromega} holds for $\upomega=c_\alpha$ given as in \eqref{eq:integralc}.
We start with the following definition.
\begin{definition}
Let $p(x)=b(|x|)$, for $b\in\Dt$. For any  $\la\in\bigl(0,b(0)\bigr)$, we set
\begin{equation}\label{eq:setTheta}
	\Theta_\la(p):=\bigl\{x\in\X:\Delta(x)\subset \{p<\la\}\bigr\},
\end{equation}
where $\Delta(x)$ is given by \eqref{eq:bigDelta} and $\{p<\la\}$ is defined as in \eqref{eq:levelset}.
\end{definition}

\begin{proposition}\label{prop:fromStoN}
Let $p(x)=b(|x|)$, for $b\in\Dt$.
Suppose that \eqref{eq:aconvomegaoveromega} holds with $\upomega=p$.
Let $c$ be given by \eqref{eq:integralc}.
Then the following analogue to \eqref{eq:aconvomegaoveromega} holds:
 \begin{equation} \label{eq:aconvomegaoveromegaanalogue}
\limsup_{\la\to0+}\sup_{x\in \Theta_\la(p)}\dfrac{( a *c)(x)}{c(x)}\leq 1.
\end{equation}
\end{proposition}
\begin{proof}
	By~Proposition~\ref{prop:new}, there is no loss of generality in assuming that $b$ is strictly decreasing on $\R_+$.
 Take an arbitrary $\delta\in(0,1)$. By \eqref{eq:aconvomegaoveromega} with $\upomega=p$, there exists $\la_0=\la_0(\delta)$ such that, for all $\la\in(0,\la_0)$, we have 
  \begin{equation}\label{eq:aconvpalphaest}
		\frac{( a *p)(x)}{p(x)}\leq 1+\delta, \quad x\in\{ p<\la\}.  
  \end{equation}
  Next, for any $x\in\X$, one gets, cf.~\eqref{eq:cx3},
  \begin{align*}
    ( a *c)(x)&=\int_\X  a (x-y)\int_{\Delta(y)} p(z)\,dz\,dy
    \\&=\int_\X  a (x-y)\int_{\Delta(x)} p(z-(x-y))\,dz\,dy=\int_{\Delta(x)} ( a *p)(z)\,dz
  \end{align*}
  As a result, by \eqref{eq:aconvpalphaest} and \eqref{eq:setTheta}, we have that, for any $x\in \Theta_\la(p)$,
  \[
  \frac{( a *c)(x)}{c(x)} = \dfrac{1}{c(x)} \int_{\Delta(x)} \frac{( a *p)(z)}{p(z)}p(z)\,dz \leq 1+\delta.
  \]
  Since the latter holds for any $\lambda\in(0,\lambda_0)$, one gets the statement.
\end{proof}

To get from \eqref{eq:aconvomegaoveromegaanalogue} the inequality \eqref{eq:aconvomegaoveromega} with $\upomega=c$, we consider the following statement.

\begin{proposition}\label{prop:ltbetter}
Let $p(x)=b(|x|)$, $x\in\X$, for a long-tailed function $b\in\Dt$.
Let $c$ be given by \eqref{eq:integralc}.
Then there exists $\la_1>0$ such that, for all $\la\in(0,\la_1)$,
\[
	\{c<\la\}\subset \Theta_\la(p).
\]
\end{proposition}
\begin{proof}
	By~Proposition~\ref{prop:new}, there is no loss of generality in assuming that $b$ is strictly decreasing on $\R_+$.
By Proposition~\ref{prop:rightsidelongtailedmeansspatially}, we have that \eqref{eq:long-tailed} holds with $c$ replaced by $p$.
As a result, for any $\eps>0$ and $r>0$, there exists $R=R(\eps,r)>0$ such that
  \[
    p(x+y)\geq (1-\eps)p(x), \quad |y|\leq r, \ |x|\geq R.
  \]
	Therefore, $x\in \{c<\la\}$ with $|x|\geq R$ implies that
  \begin{align*}
    \la&\geq \int_{x_1}^{x_1+\frac{r}{\sqrt{d}}} \ldots \int_{x_d}^{x_d+\frac{r}{\sqrt{d}}} b\Bigl(\sqrt{y_1^2+\ldots+y_d^2}\,\Bigr)\,dy_1\ldots dy_d
    \\&\geq \frac{r^d}{d^\frac{d}{2}} p\Bigl(x+\Bigl(\frac{r}{\sqrt{d}},\ldots,\frac{r}{\sqrt{d}}\Bigr)\Bigr)
    \geq \frac{r^d}{d^\frac{d}{2}} (1-\eps) p(x).
  \end{align*}
  Choose now $\eps=\frac{1}{2}$ and $r=2^{\frac{1}{d}}\sqrt{d}>0$, and consider the corresponding $R$.
	Since $\la\downarrow0$ if and only if $\m{x}\to\infty$, there exists $\la_1>0$ such that, for all $\la\in(0,\la_1)$, the inclusion $x\in \{c<\la\}$ implies $\m{x}> R$ and hence $|x|> R$.
	Moreover, for any $y\in\Delta(x)$, we have that $y\in \{c<\la\}$, by the monotonicity of $c$ in each of variables; and also we have that $\m{x}>R$ implies $|y|>R$. As a result, for any $y\in\Delta(x)$ (including $y=x$), we have that $p(y)\leq \la$, i.e $\Delta(x)\subset \{p<\la\}$.
Then, by \eqref{eq:setTheta}, $x\in \Theta_\la(p)$, that proves the statement.
\end{proof}

Combination of Propositions~\ref{prop:newsimple}, \ref{prop:twoinone}, \ref{prop:fromStoN}, \ref{prop:ltbetter} evidently implies Proposition~\ref{prop:convenienttoexplain}.

\subsubsection{\protect Proof of Theorem~\ref{thm:lin:est_above}}

\begin{proof}[\protect Proof of Theorem~\ref{thm:lin:est_above}]
Let $b_1,b_2\in\Dt$ and $b\in\Et $ satisfy the conditions of Theorem~\ref{thm:lin:est_above}, and let $c$ and $c_2$ be constructed by $b$ and $b_2$, respectively (both are defined simultaneously by either \eqref{eq:radialc} or  \eqref{eq:integralc}). By~Proposition~\ref{prop:new}, there is no loss of generality in assuming that all functions $b_1,b_2,b$ are strictly decreasing on $\R_+$. 

By Propositions~\ref{prop:newsimple}--\ref{prop:twoinone}, there exists $\alpha_1\in(0,1)$ such that, for all $\alpha\in[\alpha_1,1]$, $b^\alpha\in\Et $, and for all $\alpha\in(\alpha_1,1)$, the function $\upomega(x)=b(|x|)^\alpha$, $x\in\X$, satisfies \eqref{eq:aconvomegaoveromega}.
Choose any $\alpha_0\in\bigl(\max\bigl\{\alpha_1,\frac34\bigr\},1\bigr)$.
Let $\eps_0=\eps_0(\alpha_0)$ be given by Proposition~\ref{prop:sassaf}.
Take an arbitrary $\eps\in(0,\eps_0)$ and consider $\alpha=\alpha(\eps)\in(\alpha_0,1)$ also given by Proposition~\ref{prop:sassaf}.
Since $\log b(s)\sim \log b_2(s)$, $s\to\infty$, there exists $\rho=\rho(\alpha)=\rho(\eps)>0$ such that
\[
  -\log b_2 (s)\geq -\alpha \log b(s)>0, \quad s>\rho.
\]
Therefore, $b_2 (s)\leq b(s)^\alpha$ for $s>\rho$, and since both functions $b_2$ and $b$ are decreasing and separated from $0$ on $[0,\rho]$, there exists $B>0$ such that $b_2(s)\leq B b(s)^\alpha$, $s\in\R_+$. Let $c_\alpha$ be given by \eqref{eq:defofcalpha}. Then, clearly, $c_2(x)\leq B c_\alpha(x)$, $x\in\X$. As a result, 
\[
\|u_0\|_{c_\alpha}\leq \frac{1}{B}\|u_0\|_{c_2}<\infty.
\]

If $c$ is given by \eqref{eq:radialc}, then
by the assumed, \eqref{eq:aconvomegaoveromega} holds for $\upomega=c_\alpha$ (see \eqref{eq:defofcalpha}).

Let now $c$ be given by \eqref{eq:integralc}. Since $b$ is long-tailed, the function $b^\alpha$ is long-tailed as well. Then, one can use Proposition~\ref{prop:ltbetter} with $p$ replaced by $b^\alpha$; one gets then, for some $\la_1>0$,
  \[
		\{ c_\alpha < \lambda \} \subset \Theta_\la(p^\alpha), \quad \la\in(0,\la_1).
  \]
Therefore, Proposition~\ref{prop:fromStoN} implies that \eqref{eq:aconvomegaoveromega} holds for $\upomega=c_\alpha$.

As a result, one can use now Proposition~\ref{prop:mainabove} with $\upomega=c_\alpha$ and $\eps$ replaced by $\frac{\eps}{2}$. Namely, there exist $A_\eps>0$ and $t_0=t_0(\eps)>0$ such that
\begin{equation}\label{eq:fullineqabove11}
	\esssup_{x\notin\La(t+\frac{\eps t}{2},c_\alpha)} w(x,t)\leq \bigl(A_\eps+ B^{-1} \lVert u_0\rVert_{c_2} \bigr) e^{-\frac{\eps\beta}{4}t}, \quad t\geq t_0.
\end{equation}
On the other hand, by Proposition~\ref{prop:sassaf}, there exists $\tau=\tau(\eps)>t_0$such that \eqref{eq:supset+} holds, i.e.
\begin{equation}\label{eq:supset000}
  \X\setminus\La(t{+}\eps t,c)\subset\X\setminus \La\Bigl(t{+}\frac{\eps t}{2},c_\alpha\Bigr),\quad t\geq \tau.
\end{equation}
Combining \eqref{eq:fullineqabove11} and \eqref{eq:supset000}, one gets 
\eqref{eq:theneedeupperforlin}.
\end{proof}

\subsection{Proofs of lower estimates}\label{subsec:convtotheta}

In this Subsection, we are going to prove Theorem~\ref{thm:bb:est_below}.

Let $c$ given by $b\in\Dt$ be fixed (see Definition~\ref{def:beforeproofs}).
For any $\la>0$, we define the following function, for $x\in\X, t\geq0$,
\begin{align}
	g(x,t)&=g_{c,\eps,\la}(x,t)=\la \min\bigl\{1,c(x)e^{\beta(t{-}\eps t)}\bigr\}
	\label{eq:defofg0}\\&=\la \1_{\La(t{-}\eps t,c)}(x)+\la c(x)e^{\beta(t{-}\eps t)}\1_{\X\setminus\La(t{-}\eps t,c)}(x).\label{eq:defofg}
\end{align}
 
Define, for any $\la>0$, $\eps\in(0,1)$, and $\eta(t)$ given by \eqref{eq:explicitLa},
\begin{equation}\label{eq:hdef}
	f\left( s,t\right):=\lambda \1_{s \leq \eta\left( t{-}\eps t\right)} + \lambda e^{\beta(t{-}\eps t)}b\left( s \right)\1_{s > \eta\left( t{-}\eps t\right) }\in[0,\la], 
\end{equation}
By Definition~\ref{def:super-subexp}, for any $\eps\in(0,1)$ there exists $\tilde{\tau}=\tilde{\tau}(\eps)>0$ such that 
\[ 
	g(x,t)=f(|x|,t), \quad x\in\X,\ t\geq \tilde{\tau}.
\]
\begin{proposition} 
  Let $c=c(x)$ be given by \eqref{eq:radialc} with a long-tailed, tail-log-convex function $b\in\Dt$. Then, for any $\tau>0$, the function $f$ defined by \eqref{eq:hdef} satisfies the limit
  \begin{equation}\label{eq:uniformconvergence}
    \lim_{t\to\infty}\sup_{s\in\R_+}\biggl\lvert\frac{f(s+\tau,t)}{f(s,t)}-1\biggr\rvert=0.
  \end{equation}
\end{proposition}
\begin{proof}
	Take an arbitrary $\eps\in(0,1)$.
	For an arbitrary fixed $\tau\in\R_+$, redefine $\tilde{\tau}$ such that $\eta(\tilde{\tau}{-}\eps \tilde{\tau})\geq\tau$.
	Then, for any $t\geq \tilde{\tau}$, the function $F_{\tau,t}(s):=\frac{f(s+\tau,t)}{f(s,t)}$ takes the following values.
For $0\leq s\leq \eta(t{-}\eps t)-\tau$, one has $F_{\tau,t}(s)=1$.
For $\eta(t{-}\eps t)-\tau<s\leq \eta(t{-}\eps t)$, we have $F_{\tau,t}(s)=e^{\beta(t{-}\eps t)}b(s+\tau)$ and, since $b$ is decreasing on $[\eta(t{-}\eps t),\infty)$, one gets
\begin{align*}
\frac{b(\eta(t{-}\eps t)+\tau)}{b(\eta(t{-}\eps t))}&=e^{\beta(t{-}\eps t)}b(\eta(t{-}\eps t)+\tau)\\&\leq e^{\beta (t{-}\eps t)}b(s+\tau)
    \leq e^{\beta (t{-}\eps t)}b(\eta(t{-}\eps t))=1.
\end{align*}
Finally, for $s>\eta (t{-}\eps t)$, we have, $F_{\tau,t}(s)=\frac{b(s+\tau)}{b(s)}\leq 1$ (since $b$ is decreasing), 
\[
\frac{b(s+\tau)}{b(s)}\geq \frac{b(\eta(t{-}\eps t)+\tau)}{b(\eta(t{-}\eps t))}.
\]
  As a result, for all $s\in\R_+$,
  \begin{equation}\label{eq:usefulonH}
      0\leq 1- F_{\tau,t}(s)\leq \1_{\{s>\eta(t{-}\eps t)-\tau\}}(s)\biggl(1-\frac{b(\eta(t{-}\eps t)+\tau)}{b(\eta(t{-}\eps t))}\biggr),
  \end{equation}
  that implies the statement because of \eqref{eq:longtaileddef}.
\end{proof}

\begin{proposition}
 Let $c=c(x)$ be given by \eqref{eq:integralc} with a long-tailed, tail-log-convex function $b\in\Dt$. Let $g$ be given by \eqref{eq:defofg}. Then, for any $h\in\R_+^d$,
  \begin{equation}\label{eq:uniformlimitlongtailed}
    \lim_{t\to\infty}\sup_{x\in\X}\biggl\lvert \frac{g(x+h,t)}{g(x,t)}-1\biggr\rvert=0.
  \end{equation}
\end{proposition}
\begin{proof}
	Take an arbitrary $x\in\X$, $h\in\R_+^d$ and $t\geq \tilde{\tau}$.
	It is easy to see that $x\in\X\setminus\La(t{-}\eps t,c)$ implies $x+h\in\X\setminus\La(t{-}\eps t,c)$, and hence by monotonicity,
 \[
   \frac{g(x+h,t)}{g(x,t)}=\frac{c(x+h)}{c(x)}\leq 1.
 \]
 Let $x\in\La(t{-}\eps t,c)$. If $x+h\in\La(t{-}\eps t,c)$, then
 $ \frac{g(x+h,t)}{g(x,t)}=1$. Let now $h$ be such that $x+h\in\X\setminus\La(t{-}\eps t,c)$. Then
 \[
	 \frac{g(x+h,t)}{g(x,t)}=e^{\beta(t{-}\eps t)}c(x+h)\leq 1.
 \]
 Moreover, since $x\in\La(t{-}\eps t,c)$ implies $c(x)e^{\beta(t{-}\eps t)}\geq 1$, one has for such $x, h$ the following estimate
 \[
    0\leq 1-\frac{g(x+h,t)}{g(x,t)}\leq 1-\frac{c(x+h)}{c(x)}.
 \]
 As a result,
 \[
 \biggl\lvert \frac{g(x+h,t)}{g(x,t)}-1\biggr\rvert
 =1-\frac{g(x+h,t)}{g(x,t)}
 \leq \sup_{y:c(y+h)< e^{-\beta(t{-}\eps t)}}\biggl(1-\frac{c(y+h)}{c(y)}\biggr).
 \]
 Because of \eqref{eq:monotonelong-tailed}, for the chosen $h\in\R_+^d$ and for an arbitrary $\delta>0$, there exists $\rho=\rho(\delta,h)>0$ such that $\sup\limits_{1\leq j\leq d} y_j>\rho$ implies
 \[
 0\leq 1-\frac{c(y+h)}{c(y)}\leq \delta.
 \]
 Choose now $t_0=t_0(\rho,\eps,h)=t_0(\delta,\eps,h)\geq\tilde{\tau}$ such that
$c\bigl((\rho,\ldots,\rho)+h\bigr)>e^{-\beta(t_0{-}\eps t_0)}$. Prove that then, for any $t\geq t_0$, the inequality
 $c(y+h)\leq e^{-\beta(t{-}\eps t)}$ implies $\sup\limits_{1\leq j\leq d} y_j>\rho$. Indeed, on the contrary, suppose that, for some $t\geq t_0$, the inequality  $c(y+h)\leq e^{-\beta(t{-}\eps t)}$ holds, however, $\sup\limits_{1\leq j\leq d} y_j\leq \rho$. The latter yields
 \[
	 e^{-\beta(t{-}\eps t)}\geq c(y+h)\geq c\bigl((\rho,\ldots,\rho)+h\bigr)>e^{-\beta(t_0{-}\eps t_0)},
 \]
 that contradicts to that $t\geq t_0$. As a result, for all $x\in\X$ and $t>t_0$,
 \[
 \biggl\lvert \frac{g(x+h,t)}{g(x,t)}-1\biggr\rvert
 \leq
 \sup_{y:\sup\limits_{1\leq j\leq d} y_j> \rho}\biggl(1-\frac{c(y+h)}{c(y)}\biggr)
 <\delta,
 \]
 that implies the statement.
\end{proof}

\begin{definition}\label{def:subsollin}
Let \eqref{assum:kappa>m} hold. A function 
\[
	v\in C(\X\times[\tau,\infty)\to\R_+)\cap C^1(\X\times(\tau,\infty)\to\R_+)
\] 
is said to be a {\em sub-solution} to \eqref{eq:majorequation} on $[\tau,\infty)$ for some $\tau\geq0$, if
\begin{equation}\label{eq:Fcalmoperator}
(\F[m]  v)(x,t):=\dt v(x,t) - \ka (a*v)(x,t)+m v(x,t)\leq 0
\end{equation}
for a.a.~$x\in\X$ and for all $t\in[\tau,\infty)$.
\end{definition}

\begin{proposition}\label{prop:bb:subsollin}
Let \eqref{assum:kappa>m} hold and $c=c(x)$ be constructed by a long-tailed, tail-log-convex $b\in\Dt$. Let, for $\eps\in(0,1)$ and $\la>0$, the function $g=g(x,t)$ be given by \eqref{eq:defofg}. For a fixed $\sigma>0$, we define
\begin{equation}\label{eq:Pasha}
  v(x,t)=v_{c,\eps,\la}(x,t):=\frac{1}{\sigma} \int_{t}^{t+\sigma} g(x,s)\,ds.
\end{equation}
Then there exists $\tau_0=\tau_0(\eps)>0$ such that $v$ is a sub-solution to \eqref{eq:majorequation} on $[\tau_0,\infty)$.
\end{proposition}
\begin{proof}
Firstly, note that
\begin{align*}
\frac{\partial }{\partial t}v(x,t)&=\frac{1}{\sigma}\bigl(g(x,t+\sigma)-g(x,t)\bigr)=
\frac{1}{\sigma} \int_{t}^{t+\sigma} \frac{\partial }{\partial s}g(x,s)\,ds,\\
(a*v)(x,t)&=\frac{1}{\sigma} \int_{t}^{t+\sigma} (a*g)(x,s)\,ds,
\end{align*}
and hence
\[
  (\F[m] v)(x,t)=\frac{1}{\sigma} \int_{t}^{t+\sigma} (\F[m] g)(x,s)\,ds.
\]
Therefore, since the mapping $t\mapsto v(\cdot,t)\in E$ is continuously differentiable for $t>\tilde{\tau}$, to prove that $v$ is a sub-solution to \eqref{eq:majorequation} it is enough to show that $(\F[m] g)(x,s)\leq 0$ for a.a. $x\in\X$ and a.a. $t>\tilde{\tau}$.

We have
  \begin{align}\notag
		\frac{\partial }{\partial t}g\left( x,t\right) &=\lambda (\beta{-}\eps\beta) e^{\beta(t{-}\eps t)}b \left( \lvert x\rvert \right) \1_{\lvert x\rvert > \eta\left( t{-}\eps t\right) }\\ 
		&=(\beta-\eps \beta) g(x,t)\1_{\lvert x\rvert > \left( t{-}\eps t\right) }\leq (\beta-\eps\beta) g(x,t).\label{eq:estsubsol111}
\end{align}
Therefore, by \eqref{eq:Fcalmoperator}, \eqref{eq:estsubsol111},
  \begin{equation}
     -(\F[m]  g)  \geq \ka a* g-m g -\beta(1-\eps) g
     =\ka a* g-\ka g +\beta \eps g.\label{eq:estsubsol222}
  \end{equation}
 To find now an appropriate bound from below for $Lg=\ka a* g-\ka g$, cf. \eqref{eq:Markovgenerator}, consider two cases separately.

 1. Let $c$ by given by \eqref{eq:radialc}. 
  Since $f$ given by \eqref{eq:hdef} is decreasing in its first coordinate, we have
  \begin{align}
   \ka  ( a * g)(x,t) &=  \ka \int_\X a (-y) g(x+y,t)dy =  \ka \int_\X a (-y) f(|x+y|,t)dy\notag\\&\geq  \ka \int_\X a (-y) f(|x|+|y|,t)dy
     =  \ka \int_\X a (y) f(|x|+|y|,t)dy\nonumber\\&=
     \ka  g(x,t) \int_\X a (y)\frac{f(|x|+|y|,t)}{f(|x|,t)}dy,\label{eq:useqw}
       \end{align}
      for a.a.~$x\in\X$. Note that, by~\eqref{eq:usefulonH},
     \begin{equation}\label{eq:sdwqrfaswqfa}
			 0<\frac{f(|x|+|y|,t)}{f(|x|,t)}\leq1, \quad x,y\in\X,\ t\geq \tilde{\tau}.
     \end{equation}
     By \eqref{eq:useqw}, \eqref{eq:sdwqrfaswqfa}, and $\int_{\X} a(x) dx=1$, we have, cf.~\eqref{eq:estsubsol222},
     \begin{equation*}
     \ka  ( a * g)(x,t)-\ka  g(x,t)\geq -\ka  g(x,t) \int_\X a (y)\biggl\lvert \frac{f(|x|+|y|,t)}{f(|x|,t)}-1\biggr\rvert dy.
     \end{equation*}
Next, by \eqref{eq:uniformconvergence}, \eqref{eq:sdwqrfaswqfa}, and the dominated convergence theorem, one gets
      \begin{equation*}
        \lim_{t\to\infty} \int_\X a (y)\sup_{x\in\X}\biggl\lvert \frac{f(|x|+|y|,t)}{f(|x|,t)}-1\biggr\rvert dy=0.
      \end{equation*}
Therefore, for any $\delta\in(0,1)$ (small enough later), there exists
a $\tau_0\geq \tilde{\tau}$ such that, for all $t\geq \tau_0$ and for a.a. $x\in\X$,
  \begin{equation*}
     \ka  ( a * g)(x,t)-\ka  g(x,t)\geq -\ka  \delta g(x,t).
\end{equation*}
As a result, by \eqref{eq:estsubsol222}, 
\begin{equation*}
-\F[m] g\geq -\ka  \delta g +\beta \eps g\geq0,
\end{equation*}
if only $\delta < \dfrac{\beta \eps}{\ka}$. The proof, for $c$ given by \eqref{eq:radialc}, is fulfilled.

	2. Let $c$ be given by \eqref{eq:integralc}. Denote, for any $y\in\X$,
  \[
      y^+:=\bigl(|y_1|,\ldots,|y_d|\bigr)\in\R_+^d.
  \]
  Since the function $c$ is decreasing along all basis directions, we easily get that the function $g$ given by \eqref{eq:defofg} has the same  property (in $x$). Therefore, since $y_j\leq y^+_j$, $j=1,\ldots,d$, one gets
  \[
  g(x+y,t)\geq g(x+y^+,t).
  \]
Therefore, we will have, instead of \eqref{eq:useqw},
  \begin{align*}
    \ka ( a *g)(x,t)&= \ka \int_\X a (-y)g(x+y,t)\,dy
  \geq \ka \int_\X a (-y)g(x+y^+,t)\,dy
    \\&=\ka g(x,t) \int_\X a (y)\biggl(\frac{g(x+y^+,t)}{g(x,t)}-1\biggr)\,dy+ \ka g(x,t)\int_\X a (y)\,dy
  \end{align*}
  Taking into account \eqref{eq:uniformlimitlongtailed} for $h=y^+$, the rest of the proof is fully analogous to the first part.
\end{proof}

\begin{definition}
A function $w:\X\times\R_+\to\R_+$ is said to be a {\em sub-solution} to \eqref{eq:basicequation} on $[\tau,\infty)$ for some $\tau\geq0$, if
$(\F  w)(x,t)\leq 0$ for a.a.~$x\in\X$ and for all $t\in[\tau,\infty)$, where $\F $ is given by, cf. \eqref{eq:Fcalmoperator},
\begin{equation}\label{eq:Fcaloperator}
(\F  u)(x,t):=\dt u(x,t) - \ka (a*u)(x,t)+m u(x,t)+u(x,t)(Gu)(x,t).	
\end{equation}
\end{definition}

The proof of the following statement follows directly from Theorem~\ref{thm:comparison}.
\begin{proposition}\label{prop:subsolismajorant}
Let \eqref{assum:kappa>m}--\eqref{assum:sufficient_for_comparison}  hold. Let $0\leq u\leq\theta$ be a solution to \eqref{eq:basicequation}, and $v:\X\times\R_+\to\R_+$ be  a sub-solution to \eqref{eq:basicequation} on $[\tau,\infty)$ for some $\tau\geq0$. Suppose that, for some $t_0,t_1\geq \tau$, we have $u(x,t_0)\geq v(x,t_1)$ for a.a.~$x\in\X$. Then 
\[
u(x,t+t_0)\geq v(x,t+t_1), \quad x\in\X, \ t\geq0.
\]
\end{proposition}

We are going to find now, using the continuity of $G$ at $0$ (on $E^+$, cf.~\eqref{assum:Glipschitz})  and Proposition~\ref{prop:bb:subsollin}, sufficient conditions to have \eqref{eq:defofg} as a sub-solution to \eqref{eq:basicequation} as well.

\begin{proposition}\label{prop:bb:subsol}
Let \eqref{assum:kappa>m}--\eqref{assum:sufficient_for_comparison}  hold and $c=c(x)$ be constructed by a long-tailed, tail-log-convex function $b\in\Dt$. Then, for any $\eps\in(0,1)$, there exist $\la_0=\la_0(\eps)>0$ and $\tau_0=\tau_0(\eps)>0$,  such that, for any $\la\in[0,\la_0]$, 
the function $v=v(x,t)$, given by \eqref{eq:Pasha} and \eqref{eq:defofg},
is a sub-solution to \eqref{eq:basicequation} on $[\tau_0,\infty)$.
\end{proposition}
\begin{proof}
Take an arbitrary $\eps\in(0,1)$. For any $\delta\in(0,\eps\beta)$, one has that $m+\delta<m+\beta=\ka$; hence one can apply Proposition~\ref{prop:bb:subsollin} to the equation \eqref{eq:majorequation} with $m$ replaced by $m+\delta$. More precisely, we choose $\eps_1\in(0,1)$ to ensure that
\begin{equation}\label{eq:dsafqwe3qr325}
(\ka-(m+\delta))(1-\eps_1)=(\ka-m)(1-\eps),
\end{equation}
namely, $\eps_1:=\dfrac{\beta\eps-\delta}{\beta-\delta}$. Then, by \eqref{eq:dsafqwe3qr325} and Proposition~\ref{prop:bb:subsollin}, there exists $\tau_0=\tau_0(\eps_1)=\tau_0(\eps)$ such that
\begin{equation}\label{eq:asdsadwqerwrq}
-\F[m+\delta] v(x,t)\geq 0, \quad t\geq\tau_0,
\end{equation}
where $\F[m+\delta]$ is given by \eqref{eq:Fcalmoperator}.

Next, by \eqref{assum:kappa>m}--\eqref{assum:Glipschitz}, there exists $\la_0=\la_0(\delta)=\la_0(\eps)>0$ such that $0\leq v\leq \la_0$, $v\in E$, implies
\begin{equation}\label{eq:estsubsol10}
0\leq Gv<\delta.
\end{equation}
Clearly, \eqref{eq:defofg0} yields that $0\leq v(x,t)\leq \la$, $x\in\X$, $t\in\R_+$. Then, by \eqref{eq:Fcaloperator},
\eqref{eq:Fcalmoperator},  \eqref{eq:asdsadwqerwrq}, \eqref{eq:estsubsol10}
we have, for any $\la\in[0,\la_0]$ and for any $t\geq\tau_0$,
  \[
     -\F  v = -\F[m]v -v Gv =-\F[m+\delta]v+\delta v-v Gv
      \geq0.
  \]
The statement is proved.
\end{proof}

Now we are ready to prove Theorem~\ref{thm:bb:est_below}.

\begin{proof}[\protect Proof of Theorem~\ref{thm:bb:est_below}] 
Recall that, by Theorem~\ref{thm:existandcompared},  $0\leq u_0\leq \theta$ implies $0\leq u(\cdot,t)\leq \theta$ for $t>0$; and then, by \eqref{assum:Gpositive}, $Gu\leq\beta$. Rewrite \eqref{eq:basicequation} in the form \eqref{eq:RDequation} with $F$ given by \eqref{eq:FthroughG}, then, by \eqref{eq:propertiesofF}, $Fu\geq0$.
Therefore, for all $t>0$ and a.a.~$x\in\X$,
\begin{align*}
u(x,t)&=e^{-\ka t}u_0(x)+\ka\int_0^t e^{-\ka( t-s)}  (a*u)(x,s)ds+\int_0^t e^{-\ka (t-s)} (Fu)(x,s) ds\\&\geq e^{-\ka t}u_0(x)+\ka\int_0^t e^{-\ka( t-s)}(a*u)(x,s)ds.
\end{align*}
The same inequality for $u(x,s)$ implies
\begin{align}
u(x,t)&\geq \ka\int_0^t e^{-\ka( t-s)}      (a*u)(x,s)ds\geq
\ka\int_0^t e^{-\ka( t-s)}  e^{-\ka s}(a*u_0)(x)ds\notag        \\&=
\ka t e^{-\ka t}(a*u_0)(x)\geq \ka t    e^{-\ka t}c(x),\label{eq:coolest}
\end{align}
for all $t\geq0$ and a.a.~$x\in\X$, because of \eqref{eq:condonparambelow}.

Fix an arbitrary $\eps\in(0,1)$. Take any $\delta\in(0,\eps)$ and consider $\la_0=\la_0(\delta)>0$ and 
 $\tau_0=\tau_0(\delta)> \sigma$,
both given by~Proposition~\ref{prop:bb:subsol}.
Set now
\[
	\la:=\min\bigl\{\la_0, \ka \tau_0 e^{-(\ka+(\beta{-}\delta\beta)) \tau_0}\bigr\}.
\]
Then, by \eqref{eq:coolest} and \eqref{eq:defofg0}, we have, for a.a.~$x\in\X$,
\begin{equation}\label{eq:onemoreeq}
	u(x,\tau_0)\geq \la e^{(\beta{-}\delta\beta) \tau_0}c(x)\geq
\la \min\bigl\{e^{(\beta-\delta\beta) \tau_0}c(x),1\bigr\}=
g_{c,\delta,\la}(x,\tau_0).
\end{equation}
Next, the function $g_{c,\delta,\la}(x,t)$ is non-decreasing in $t$, hence \eqref{eq:Pasha} yields
\begin{equation}\label{eq:coool}
 g_{c,\delta,\la}(x,t)\leq v_{c,\delta,\la}(x,t)\leq g_{c,\delta,\la}(x,t+\sigma),
\end{equation}
and hence we can continue \eqref{eq:onemoreeq} as follows:
\[
  u(x,\tau_0)\geq v_{c,\delta,\la}(x,\tau_0- \sigma), \quad \text{for a.a. } x\in\X.
\]
Therefore, by Propositions~\ref{prop:bb:subsol} and \ref{prop:subsolismajorant}, one gets, for any $\tau\geq0$ and a.a. $x\in\X$
\begin{equation*}
  u(x,\tau_0+\tau)\geq v_{c,\delta,\la}(x,\tau_0- \sigma +\tau)\geq g_{c,\delta,\la}(x,\tau_0- \sigma +\tau),
\end{equation*}
where the latter inequality is because of \eqref{eq:coool}.
As a result,
\begin{equation}\label{eq:ineqdopevident}
  u(x,\tau_0+\sigma+\tau)\geq \la \quad  \text{for a.a.} \ x\in\La((1-\delta)(\tau_0+\tau),c), \quad \tau\geq0.
\end{equation}

By Proposition~\ref{prop:new}, without loss of generality we may assume that $c$ is given by a strictly decreasing on $\R_+$ function. We will distinguish two cases.

1. Let $c$ be given by \eqref{eq:radialc}. Fix $\tau\geq0$. Since \eqref{eq:explicitLa} holds, we have that the set
\begin{align*}
	\widetilde{\Lambda}:=&\{y\in\X: B_1(y)\subset\La((1-\delta)(\tau_0+\tau),c)\} \\
		=&\bigl\{y\in\X: B_1(y)\subset B_{\eta((1-\delta)(\tau_0+\tau),b)}(0)\bigr\}
\end{align*}
is nothing but $B_{\eta_\delta^-(\tau_0+\tau,b)-1}(0)$ and, moreover,
\begin{equation}\label{eq:evidentunion}
  \La((1-\delta)(\tau_0+\tau),c)=\bigcup_{y\in \widetilde{\Lambda}} B_1(y).
\end{equation}
Take and fix now an arbitrary $y\in\widetilde{\La}$, i.e. $|y|\leq \eta((1-\delta)(\tau_0+\tau))-1$.
Then, by~\eqref{eq:ineqdopevident},
\[
u(x,\tau_0+\sigma+\tau) \geq \la \1_{B_1(y)}(x) \quad \text{for a.a. } x\in\X.
\]
Consider now equation \eqref{eq:basicequation} with the initial condition $v_0(x)= u(x,\tau_0+\sigma+\tau)$, $x\in\X$; let $v(x,t)$ be the corresponding solution to \eqref{eq:basicequation}.
By the uniqueness in Theorem~\ref{thm:existandcompared}, $v(x,t)= u(x,\tau_0+\sigma+\tau+t)$, $t\in\R_+$.

Take an arbitrary $\mu\in(0,\theta)$. Apply Theorem~\ref{thm:hairtrigger} to the solution $v$ and $K=B_1(y)$; then there exists $t_\mu\geq 1$ such that $v(x,t)\geq \mu$ for a.a. $x\in B_1(y)$.
As a result, 
\begin{equation}\label{eq:ineqonemore}
u(x,\tau_0+\sigma+t_\mu+\tau)
\geq \mu,
\end{equation}
for all $\tau\geq0$ and a.a.~$x\in B_1(y)$. 
Stress that $t_\mu$ does not depend on a $y$ with $|y|\leq \eta((1-\delta)(\tau_0+\tau))-1$.
As a result, by \eqref{eq:evidentunion} for any $\delta\in(0,1)$ and $\mu\in (0,\theta)$, there exist $\la_0=\la_0(\delta)>0$, $\tau_0=\tau_0(\delta)>0$, and $t_\mu\geq 1$ such that, for all $\tau\geq 0$ and for a.a.~$x$ with $|x|\leq \eta((1-\delta)(\tau_0+\tau))$, the inequality \eqref{eq:ineqonemore} holds.

By the definition of $\eta$ (see \eqref{eq:explicitLa}), one gets that there exists $\tau_1\geq0$ such that, for all $\tau\geq \tau_1$,
\[
\eta((1-\eps)(\tau+\tau_0+\sigma+t_\mu))
\leq\eta((1-\delta)(\tau+\tau_0)),
\]
i.e. \eqref{eq:ineqonemore} holds for all $\tau\geq \tau_1$ and a.a.~$x$ with $|x|\leq \eta((1-\eps)(\tau+\tau_0+\sigma+t_\mu))$. Since $\mu\in(0,\theta)$ was arbitrary, the latter fact yields \eqref{eq:whatweareproving1}.

2. Let now $c$ be given by \eqref{eq:integralc}. Consider the norm on $\X$ given by
\[
|x|_\infty:=\lvert(x_1,\ldots,x_d)\rvert_\infty:=\max_{1\leq j\leq d}|x_j|.
\]
Let $\widetilde{B}_{\frac{1}{2}}(x)$ denote the ball with the center at an $x\in\X$ and the radius $\frac{1}{2}$ w.r.t.~the $|\cdot|_\infty$-norm. Then, clearly,
\[
\widetilde{B}_{\frac{1}{2}}(x)=\bigtimes_{j=1}^d\Bigl[x_j-\frac12,x_j+\frac12\Bigr]
=\bigtimes_{j=1}^d\bigl[y_j-1,y_j\bigr]=:C_1(y),
\]
where $y_j=x_j+\frac{1}{2}$, $1\leq j\leq d$.
For $y\in \La((1-\delta)(\tau_0+\tau),c)$,
\[
C_1(y)\subset\La((1-\delta)(\tau_0+\tau),c).
\]
Therefore, cf.~\eqref{eq:evidentunion},
\[
 \La((1-\delta)(\tau_0+\tau),c) =\bigcup_{y\in  \La((1-\delta)(\tau_0+\tau),c)} C_1(y).
\]
Hence, one can just repeat the previous proof, applying Theorem~\ref{thm:hairtrigger} to the solution $v$ and $K=C_1(y)$ with $y\in\La((1-\delta)(\tau_0+\tau),c)$.
\end{proof}

\subsection{Proofs of general results}\label{subsec:corandex}
We are going to prove the main Theorem~\ref{thm:combined}. 
Consider separately proofs for the items 1(a)--(b) and 2(a)--(b).

\begin{proof}[Proof of Theorem~\ref{thm:combined}, item 1(a)]
Let $\eps_0\in(0,1)$ be chosen later. Take an arbitrary $\eps\in(0,\eps_0)$.

Let $c_+ \in L^1(\X)$ be constructed by a long-tailed, tail-log-convex function $b_+\in\Dt$.
Note that \eqref{eq:asdsdps} yields $u_0\in L_1(\X)$.
Therefore, one can apply Proposition~\ref{prop:liminfbelow} with $c=c_+>0$ and $f=u_0$; namely, there exists $D>0$ such that
  $ a *u_0\geq c_+*u_0\geq D c_+$. Then, by Theorem~\ref{thm:bb:est_below}, the convergence \eqref{eq:whatweareproving1} holds, with $\eps$ replaced by $\frac{\eps}{2}<\eps_0$ and $c$ replaced by $Dc_+$. Since the functions $Db_+$ and $b$ are also log-equivalent, one can apply Proposition~\ref{prop:wehavethis} with $b_1=b$ and $b_2=Db_+$, to get inclusion
	$\La(t{-}\eps t,c)\subset\La(t{-}\frac{\eps t}{2},Dc_+)$. As~a~result, \eqref{eq:whatweareproving1} holds, with $c(x)=b(|x|)$, $x\in\X$. Note that we had not any restrictions on $\eps_0$ here.

	Since $b^+\in\Dt$, we can apply Theorem~\ref{thm:lin:est_above} with $b_1=b_2=b^+$ and the given $b\in\Et $. Indeed, \eqref{condB1} implies \eqref{ass:newA5}, and, for the $c_2$ constructed by $b^+$ and satisfying \eqref{eq:radialc}, \eqref{eq:asdsdps} is just \eqref{eq:u0leqb2}. Therefore, \eqref{eq:theneedeupperforlin} holds, that, we recall, implies \eqref{eq:whatweareproving2} because of \eqref{eq:majorisedbylinear}.
  \end{proof}

\begin{proof}[Proof of Theorem~\ref{thm:combined}, item 1(b)]
The proof of \eqref{eq:whatweareproving1} is essentially the same as that for the item 1(a), with only the  difference that we will apply now Proposition~\ref{prop:liminfbelow} for $c=v_\circ>0$ and $f= a \in L^1(\X)$.
Next, since $b^\circ\in\Dt$ and \eqref{condB4} holds, we can apply Theorem~\ref{thm:lin:est_above} with $b_1=b_2=b^\circ$.
\end{proof}

\begin{proof}[Proof of Theorem~\ref{thm:combined}, item 2(a)]
Let $\eps_0\in(0,1)$ be chosen later. Take an arbitrary $\eps\in(0,\eps_0)$. 
By~\eqref{condB1} and \eqref{eq:hevisidelb}, we have
   \begin{align*}
( a *u_0)(x)&\geq \zeta \int_\X b_+(|y|)\1_{\R_-^d}(x-y)\,dy\notag
\\&=\zeta \int_{\Delta(x)} b_+(|y|)\,dy=:\tilde{c}(x), \quad x\in\X.
\end{align*}
One can apply Theorem~\ref{thm:bb:est_below} 
  to get \eqref{eq:whatweareproving1} with $c$ replaced by $\tilde{c}$ and $\eps$ replaced by $\frac{\eps}{2}$. Since the functions $b$ and $\zeta b_+$ are log-equivalent, one can apply Proposition~\ref{prop:wehavethis} with  $c^{(1)}(x)=c(x):=\int_{\Delta(x)} b(|y|)dy$ and $c^{(2)}(x)=\tilde{c}(x)$, $x\in\X$; and then \eqref{eq:supsetlog-} leads to \eqref{eq:whatweareproving1} for this $c$.

	To get \eqref{eq:whatweareproving2} we will need just to repeat all corresponding arguments from the proof of the item 1(a) with only the difference that Theorem~\ref{thm:lin:est_above}  will be applied now for functions satisfying \eqref{eq:integralc}.
\end{proof}

\begin{remark}
Using \cite[Proposition~5.4 (Q2)]{FT2017a} and modifying accordingly the proof of  Theorem~\ref{thm:bb:est_below}, one can replace $\R_-^d$ in \eqref{eq:hevisidelb} by $\bigtimes\limits_{j=1}^d(-\infty,\overline{y}_j]$, for an arbitrary fixed $\overline{y}\in\X$.
\end{remark}

\begin{remark}
If, additionally, $u_0(x)=\int_{\Delta(x)}p(y)dy$, $x\in\X$ for some $p\in L^1(\X)$, then, evidently, 
\[
\sup\limits_{x\in\X} \dfrac{p(x)}{ a (x)}<\infty\quad \Longrightarrow \quad 
\sup\limits_{x\in\X}\dfrac{u_0(x)}{\displaystyle\int_{\Delta(x)} a (y)dy}<\infty.
\]
\end{remark}

\begin{proof}[Proof of Theorem~\ref{thm:combined}, item 2(b)]
  First,
  we apply Proposition~\ref{prop:liminfbelow} with $f= a $ and $c$ replaced by $b_\circ$. Then, similarly to the proof of the item~2(a), we may apply Theorem~\ref{thm:bb:est_below} to get \eqref{eq:whatweareproving1} with $c$ replaced by $v_\circ$ and $\eps$ replaced by $\frac{\eps}{2}$, and, by using the log-equivalence between $b$ and $b_\circ$ and Proposition~\ref{prop:wehavethis}, we will get \eqref{eq:whatweareproving1} for the required $c$.

  To get \eqref{eq:whatweareproving2}, one can use the same arguments as in the proof of the item~1(b).
\end{proof}

\appendix
\renewcommand{\theequation}{\thesection{p}.\arabic{equation}}
\section{Appendix}

\begin{proof}[Proof of Lemma~\ref{le:GisOK}]
Firstly, we note that \eqref{eq:sepfromzerobeta} implies \eqref{assum:a_nodeg} with $\rho=\frac{\varrho}{\max\{\ka,1\}}$. Let $G$ be defined by \eqref{eq:defGbyF}, i.e., for $0\leq u\in E$ and $x\in\X$,
  \[
    (Gu)(x) = \beta - \alpha \frac{f\bigl(u(x)\bigr)}{u(x)} - (1-\alpha) \beta \Bigl(1-\frac{(a^-*u)(x)}{\theta}\Bigr)^k, 
  \]
where $\frac{f(s)}{s}:=\beta$ for $s=0$. Then it is straightforward to check that \eqref{assum:Gpositive}--\eqref{assum:Glipschitz} and \eqref{assum:G_locally_continuous}--\eqref{assum:G_increas_on_const} hold.
We are going to prove that there exists $p\geq0$ such that, for any $v,w\in E_\theta^+$ with $v\leq w$, 
\begin{align}\notag
p(w-v) &+ \ka a*(w-v)\\
& \geq (w-v)Gw + v(Gw-Gv)  + \varrho \1_{B_\varrho(0)}*(w-v). \label{eq:reinfdop}
\end{align}
Note that \eqref{eq:reinfdop} evidently implies \eqref{assum:sufficient_for_comparison}.
Next, \eqref{assum:improved_sufficient_for_comparison} will follow from \eqref{eq:reinfdop} if we choose any $\delta<\varrho$ with $\ka \delta<\varrho$ and any $b\in C^\infty(\X)\cap L^\infty(\X)$, such that
$\ka a -\varrho \1_{B_\varrho(0)} \leq \ka b \leq \ka a -\ka \delta\1_{B_\delta(0)}$.

By  \eqref{eq:coninlocalcase}, there exists a Lipschitz constant $K>0$, such that
\begin{align}
Gw-Gv&=\alpha \Bigr(\frac{f(v)}{v}-\frac{f(w)}{w}\Bigr)+ 
(1-\alpha)\beta \biggl[\Bigl(1-\frac{a^-*v}{\theta} \Bigr)^k - \Bigl(1-\frac{a^-*w}{\theta} \Bigr)^k \biggr]\notag\\
&\leq \alpha K (w-v)+(1- \alpha)\beta k\frac{a^-*(w-v)}{\theta},\label{eq:keepit}
\end{align}
where we used an elementary inequality $q^k-r^k\leq k(q-r)$ for $0\leq r\leq q\leq 1$. Multiplying both parts of \eqref{eq:keepit} on $0\leq v\leq\theta$ and using \eqref{eq:sepfromzerobeta}, we get
\[
  v(Gw-Gv)\leq \alpha \theta K (w-v)+\ka a*(w-v) - \varrho \1_{B_\varrho(0)}*(w-v).
\]
Finally, by \eqref{assum:Gpositive}, $(w-v)Gw\leq \beta(w-v)$, and therefore, the inequality \eqref{eq:reinfdop} holds with $p:= \beta+\alpha \theta K>0$.
\end{proof}

\begin{lemma}\label{le:almostlinear}
Let $\la>1$ and let $b:\R_+\to\R_+$ be defined, for large $s$, as follows
\[
        b(s)=\exp\Bigl(-\frac{s}{(\log s)^\lambda}\Bigr).
\]
Let $\beta>0$, and define, for large $t$, the function $\eta(t):=b^{-1}\bigl(e^{-\beta t}\bigr)$. Then
\begin{equation}\label{eq:lambertetc}
         \eta(t)\sim \beta t (\log t)^\lambda, \quad t\to\infty.
    \end{equation}
\end{lemma}
\begin{proof}
The equation $b(s)=e^{-\beta t}$ yields $s(\log s)^{-\lambda}=\beta t$. Making substitution $s=e^{\tau}$, one easily gets
   \[
   -\frac{\tau}{\lambda} e^{-\frac{\tau}{\lambda}}=-\frac{1}{\lambda(\beta t)^{\frac{1}{\lambda}}}.
   \]
   Since $s>e^\lambda$ implies $-\frac{\tau}{\lambda}<-1$ and assuming $t$ big enough, to ensure that $-\frac{1}{\lambda(\beta t)^{\frac{1}{\lambda}}}>-\frac{1}{e}$, one has that the solution to the latter equation can be given in terms of the negative real branch $W_{-1}$ of Lambert W-function, that is the function such that $W_{-1}(\nu)\exp(W_{-1}(\nu))=\nu$, $W_{-1}(\nu)<-1$, $\nu\in (-e^{-1},0)$. Namely, one gets
   $
   -\frac{\tau}{\lambda}=W_{-1}\bigl(-\lambda^{-1}(\beta t)^{-\frac{1}{\lambda}}\bigr),
   $
   and, therefore
   \[
   \eta(t)=\exp\biggl( -\lambda W_{-1}\Bigl(-\frac{1}{\lambda(\beta t)^{\frac{1}{\lambda}}}\Bigr)  \biggr).
   \]
   However, $\exp(-W_{-1}(\nu)) =\nu^{-1}W_{-1}(\nu)$, therefore,
   \[
   \exp(-\lambda W_{-1}(\nu)) =(-\nu)^{-\lambda}(-W_{-1}(\nu))^\lambda,
   \]
   i.e.
   \[
   \eta(t)=\lambda^\lambda \beta t\biggl( -W_{-1}\Bigl(-\frac{1}{\lambda(\beta t)^{\frac{1}{\lambda}}}\Bigr)  \biggr)^\lambda, \quad t>\frac{1}{\beta}\Bigl(\frac{e}{\lambda}\Bigr)^\lambda.
   \]
   It is well-known that $W_{-1}(\nu)\sim \log(-\nu)$, $\nu\to0-$. This yields \eqref{eq:lambertetc}.
\end{proof}

\begin{lemma}\label{le:2dim}
Let a function $X(t)\to\infty$, $t\to\infty$, be such that, for large $t$,
\begin{equation}\label{eq:defetaC2}
\int_{X(t)}^\infty\int_{X(t)}^\infty b(|y|)\,dy_1\,dy_2=e^{-\beta t}, \qquad |y|=\sqrt{y_1^2+y_2^2},
\end{equation}
where $\beta>0$ and $b:\R_+\to\R_+$ is a decreasing at $\infty$ function, such that $\int_{\R_+} b(r)r\,dr<\infty$. Consider the following functions 
\begin{equation}\label{eq:cxmut}
  c(x):=\frac{\pi}{2}\int_{\sqrt{2}x}^\infty b(r)r\,dr, \qquad
  \mu(t):=c^{-1}\bigl(e^{-\beta t}\bigr)
\end{equation}
for large $x$ and $t$. 
Then, for any $\eps\in(0,1)$ and large $t$,
\[
  \mu(t)\geq X(t) \geq \frac{1}{2}\mu(t-\eps t). 
\]
\end{lemma}
\begin{proof}
Rewriting the set $\{(y_1,y_2)\in\R^2\mid y_1\geq X(t), y_2\geq X(t)\}$ for $X(t)>0$ in polar coordinates, we obtain from \eqref{eq:defetaC2} that, for large $t$,
\begin{align*}
 e^{-\beta t}&=\int_{\sqrt{2}X(t)}^\infty\int_{\arcsin\frac{X(t)}{r}}^{\arccos\frac{X(t)}{r}} b(r)r\,dr
  =\int_{\sqrt{2}X(t)}^\infty\biggl(\frac{\pi}{2}-2\arcsin\frac{X(t)}{r}\biggr) b(r)r\,dr\\
  &=X(t)^2\int_{\sqrt{2}}^\infty\biggl(\frac{\pi}{2}-2\arcsin\frac{1}{s}\biggr) b\bigl(X(t)s\bigr)\,s\,ds.
\end{align*}
Therefore, for any $\delta>0$,
\begin{align}
c\bigl( X(t)\bigr)\geq e^{-\beta t}&\geq X(t)^2\int_{\sqrt{2}+\delta}^\infty\biggl(\frac{\pi}{2}-2\arcsin\frac{1}{s}\biggr) b\bigl(X(t)s\bigr)\,s\,ds\notag\\
&\geq f(\delta) X(t)^2 \int_{\sqrt{2}+\delta}^\infty b\bigl(X(t)s\bigr)\,s\,ds=\frac{2}{\pi} f(\delta) c\biggl(\frac{\sqrt{2}+\delta}{\sqrt{2}}X(t) \biggr),\label{eq:ssfarsaddsasa}
\end{align}
where 
\[
  f(\delta):=\frac{\pi}{2}-2\arcsin\frac{1}{\sqrt{2}+\delta}\in\Bigl(0,\frac{\pi}{2}\Bigr), \quad \delta>0,
\]
is an increasing function. Since $c(x)$ is decreasing, we obtain from \eqref{eq:ssfarsaddsasa} that
\begin{equation}\label{eq:sassdwwee2}
c^{-1}(e^{-\beta t})\geq X(t) \geq \frac{\sqrt{2}}{\sqrt{2}+\delta}c^{-1}\biggl(\frac{\pi}{2f(\delta)}e^{-\beta t}\biggr).
\end{equation}
Set $\la=\frac{\sqrt{2}}{\sqrt{2}-1}>1$. Choose $\delta>0$ such that
\[
  f(\delta)=\frac{\pi}{2\la}<\frac{\pi}{2},
\]
then 
\[
  \frac{\sqrt{2}}{\sqrt{2}+\delta}=\sqrt{2}\sin\biggl(\frac{\pi}{4}\Bigl(1-\frac{1}{\la}\Bigr)\biggr)>\frac{1}{\sqrt{2}}\Bigl(1-\frac{1}{\la}\Bigr)=\frac{1}{2},
\]
where we used the inequality $\sin x>\frac{2}{\pi}x$ for $0<x<\frac{\pi}{2}$. Then \eqref{eq:sassdwwee2} implies
\[
  \mu(t)\geq X(t) \geq \frac{1}{2}c^{-1}\bigl(\la e^{-\beta t}\bigr).
\]
Take finally an $\eps\in(0,1)$ and assume that $t$ is big enough to ensure that $e^{\eps\beta t}>\la$. Since $c^{-1}(x)$ is a decreasing function, one gets the statement.
\end{proof}

\begin{remark}\label{rem:forweibR2ex}
Let \eqref{eq:weibexmon1}--\eqref{eq:weibexmon2} holds. Then, by Theorem~\ref{thm:combined}, \eqref{eq:whatweareproving1}--\eqref{eq:whatweareproving2} hold with $\La(t)=\La(t,c)$ given by \eqref{eq:defLa} where $c(x_1,x_2)=\int_{x_1}^\infty\int_{x_2}^\infty b(|y|)\,dy_1\,dy_2$, cf. \eqref{eq:frontinR2}, and $b\in\Et$ is log-equivalent to $e^{-\sqrt{s}}$, $s>0$. Take $b(s)=\frac{1}{\pi} s^{-\frac{3}{2}}e^{-\sqrt{s}}$ for large $s$. By~\cite[Corollary 3.1]{FT2017b}, $b\in\Et$. Let $X(t):=X_1(t)=X_2(t)$ describe the motion of the boundary of $\La(t)$ in the diagonal direction in \eqref{eq:frontinR2}. Then, by Lemma~\ref{le:2dim}, we have, cf.~\eqref{eq:cxmut},
\[
  c(x)=\frac{\pi}{2}\frac{1}{\pi} \int_{\sqrt{2}x}^\infty \frac{1}{\sqrt{r}}\exp\bigl(-\sqrt{r}\bigr)\,dr
  =\exp\bigl(-\sqrt[4]{2}\sqrt{x}\bigr).
\]
Then, by \eqref{eq:cxmut}, $\mu(t)=\frac{\beta^2}{\sqrt{2}}t^2$. Therefore, by Lemma~\ref{le:2dim}, for any $\eps\in(0,1)$ and large $t$, \eqref{eq:weibestforeta} holds.
\end{remark}

 
\newpage

\def\cprime{$'$}

\end{document}